\documentclass{article}%
\usepackage{amsmath}
\usepackage{amsfonts}
\usepackage{amssymb}
\usepackage{graphicx}
\usepackage{subfigure}
\usepackage{subfloat}

\providecommand{\U}[1]{\protect\rule{.1in}{.1in}}
\newtheorem{theorem}{Theorem}

\newtheorem{corollary}[theorem]{Corollary}

\newtheorem{definition}[theorem]{Definition}

\newtheorem{lemma}[theorem]{Lemma}

\newenvironment{proof}[1][Proof]{\noindent\textbf{#1.} }{\ \rule{0.5em}{0.5em}}
\input{tcilatex}

\begin{document}

\title{Convergence rate of weak Local Linearization schemes for stochastic
differential equations with additive noise}

\author{J.C. Jimenez
\thanks{Instituto de Cibernetica, Matematica y Fisica, Calle 15,
No. 551, entre C y D, Vedado, La Habana 10400, Cuba. e-mail:
jcarlos@icimaf.cu} \ and F. Carbonell
\thanks{Research Scientist, Biospective Inc., Montreal, Canada. e-mail:
felixmiguelc@gmail.com}}

\maketitle

\begin{abstract}
There exists a diversity of weak Local Linearization (LL) schemes
for the integration of stochastic differential equations with
additive noise, which differ with respect to the algorithm that is
employed in the numerical implementation of the weak Local Linear
discretizations. On the contrary to the Local Linear discretization,
the rate of convergence of the LL schemes has not been considered up
to now. In this work, a general theorem about this issue is derived
and further is applied to a number of specific schemes. As
application, the convergence rate of weak LL schemes for equations
with jumps is also presented.
\end{abstract}

\section{Introduction}

The evaluation of Wiener functional space integrals and the
estimation of difussion processes are essential matters for the
resolution of a number of problems in mathematical physics, biology,
finance and other fields. In the solution of this kind of problems,
weak numerical integrators for Stochastic Differential Equations
(SDEs) have become an important tool \cite{Talay90, Talay90b,
Kloeden 1995, Rao99, Jimenez06 APFM}. Well-known are, for instance,
the Euler, the Milstein, the Talay-Tubaro extrapolation, the
Runge-Kutta and the Local Linearization methods (see \cite{Schurz
2002} for a review of these methods).

Specifically, weak Local Linearization (LL) schemes for SDEs with
additive noise have played a prominent role in the construction of
effective inference methods for SDEs \cite{Shoji 1997, Shoji 1998b,
Durham02, Singer02, Hurn07}, in the estimation of distribution
functions in Monte Carlo Markov Chain methods \cite{Stramer99b,
Roberts01, Hansen03} and the simulation of likelihood functions
\cite{Nicolau02}. Extensive simulation studies carried out in these
papers have showed that these LL schemes posses high numerical
stability and remarkable computational efficiency. Other distinctive
feature of the weak LL integrators is that they preserve the
ergodicity and geometric ergodicity properties of a wide class of
nonlinear SDEs \cite{Hansen03}.

This paper deals with an open problem related with the weak Local
Linearization method. It is known that the order-$\beta$ weak Local
Linear discretization is the base for the construction and study of
such a method \cite{Carbonell06}. Starting with this discretization
a variety of numerical schemes can be derived, which mainly differ
with respect to the algorithm used in the numerical implementation
of the discretization. This feature provides flexibility to the LL
method for suitable adjustments when is applied to certain types of
equations (e.g., large systems of SDEs, etc). However, in contrast
with the weak Local Linear discretization, the convergence rate of
the schemes has not been considered until now. This essential issue
must be addressed for developing computationally efficient weak LL
schemes.

In this work, a main theorem about the convergence rate of weak LL
schemes for SDEs with additive noise is derived and, on this base,
convergence results are obtained for some specific schemes. As
direct application, the convergence rate of some weak LL schemes for
equations with jumps is also demostrated. A summary of basic results
on the LL method is presented for supporting the subsequent
presentation.

\section{Notations and preliminaries\label{Section WLL aprox}}

Let $(\Omega,\mathcal{F},P)$ be a complete probability space, and
$\{\mathcal{F}_{t},$ $t\geq t_{0}\}$ be an increasing right
continuous family of complete sub $\sigma$-algebras of
$\mathcal{F}$. Consider a $d$-dimensional diffusion process
$\mathbf{x}$ defined by the following stochastic
differential equation with additive noise%
\begin{align}
d\mathbf{x}(t)  & =\mathbf{f}(t,\mathbf{x}(t))dt+\sum\limits_{i=1}%
^{m}\mathbf{g}_{i}(t)d\mathbf{w}^{i}(t)\text{ \ }\label{WSDE-LLA-1}\\
\mathbf{x}(t_{0})  & =\mathbf{x}_{0},\label{WSDE-LLA-2}%
\end{align}
where the drift coefficient $\mathbf{f}:[t_{0},T]\times\mathbb{R}%
^{d}\rightarrow$ $\mathbb{R}^{d}$ and the diffusion coefficient $\mathbf{g}%
_{j}:\left[  t_{0},T\right]  \rightarrow$ $\mathbb{R}^{d}$ are
differentiable functions,
$\mathbf{w=(\mathbf{w}}^{1},\ldots,\mathbf{w}^{m}\mathbf{)}$ is an
$m$-dimensional $\mathcal{F}_{t}$-adapted standard Wiener process,
and $\mathbf{x}_{0}$ is a $\mathcal{F}_{t_{0}}$-measurable random
vector. The standard conditions for the existence and uniqueness of
a solution for (\ref{WSDE-LLA-1})-(\ref{WSDE-LLA-2}) are assumed.

Consider the time discretization $(t)_{h}=\{t_{n}:n=0,1,\ldots,N\}$,
with maximum step-size $h\in(0,1)$, defined as a sequence of
$\mathcal{F}$-stopping
times that satisfy $t_{0}<t_{1}<\cdots<t_{N}=T$ and $\sup\limits_{n}%
(h_{n})\leq h$, w.p.1, where $t_{n}$ is
$\mathcal{F}_{t_{n}}$-measurable for each $n=0,1,\ldots,N$, and
$h_{n}=t_{n+1}-t_{n}$. In addition, let us denote
$n_{t}=\max\{n=0,1,2,\ldots:$ $t_{n}\leq t$ and $t_{n}\in\left(
t\right) _{h}\}$ for all $t\in\lbrack t_{0},T]$.

\subsection{Weak Local Linear discretization \cite{Carbonell06}}

\begin{definition}
For a given time discretization $\left(  t\right)  _{h},$ the
order-$\beta$
$(=1,2)$ weak Local Linear discretization of the solution of (\ref{WSDE-LLA-1}%
)-(\ref{WSDE-LLA-2}) is defined by the recurrent relation
\begin{equation}
\mathbf{y}_{n+1}=\mathbf{y}_{n}+\mathbf{\phi}_{\beta}(t_{n},\mathbf{y}%
_{n};h_{n})+\mathbf{\eta}(t_{n},\mathbf{y}_{n};h_{n}),\label{LLDiscretization}%
\end{equation}
where
\begin{equation}
\mathbf{\phi}_{\beta}(t_{n},\mathbf{y}_{n};h_{n})=\int\limits_{0}^{h_{n}%
}e^{\mathbf{f}_{\mathbf{x}}(t_{n},\mathbf{y}_{n})(h_{n}-s)}(\mathbf{f}%
_{\mathbf{x}}(t_{n},\mathbf{y}_{n})\mathbf{y}_{n}+\mathbf{a}_{n}%
^{\mathbb{\beta}}(t_{n}+s))ds,\label{DeterministicLL}%
\end{equation}
and $\mathbf{\eta}\left(  t_{n},\mathbf{y}_{n};h\right)  $ is a zero
mean Gaussian random variable with variance
\begin{equation}
\mathbf{\Sigma}(t,\mathbf{y};\delta)=\int\limits_{0}^{\delta}e^{\mathbf{f}%
_{x}(t,\mathbf{y})(\delta-s)}\mathbf{G}(t+s)\mathbf{G}^{\intercal
}(t+s)e^{\mathbf{f}_{\mathbf{x}}^{\intercal}(t,\mathbf{y})(\delta
-s)}ds.\label{CovarianceMatrix}%
\end{equation}
Here, $\mathbf{G}(u)\mathbf{=[g}_{1}(u),\ldots,\mathbf{g}_{m}(u)]$
is an $d\times m$ matrix,
\[
\mathbf{a}_{n}^{\mathbb{\beta}}\left(  u\right)  =\left\{
\begin{array}
[c]{cc}%
\mathbf{f(}t_{n},\mathbf{y}_{n})-\mathbf{f}_{\mathbf{x}}(t_{n},\mathbf{y}%
_{n})\mathbf{y}_{n}+\mathbf{f}_{t}(t_{n},\mathbf{y}_{n})(u-t_{n}) &
\text{for
}\mathbb{\beta}=1\\
\mathbf{a}_{n}^{1}\left(  u\right)  +\frac{1}{2}\sum\limits_{j=1}%
^{m}(\mathbf{I}_{d\times d}\otimes\mathbf{g}_{j}^{\intercal}\left(
t_{n}\right)  )\mathbf{f}_{\mathbf{xx}}(t_{n},\mathbf{y}_{n})\mathbf{g}%
_{j}\left(  t_{n}\right)  \left(  u-t_{n}\right)   & \text{for
}\mathbb{\beta }=2
\end{array}
\right.  ,
\]
$\mathbf{f}_{\mathbf{x}}$, $\mathbf{f}_{t}$ denote the partial
derivatives of $\mathbf{f}$ with respect to the variables
$\mathbf{x}$ and $t$, respectively, $\mathbf{f}_{\mathbf{xx}}$ the
Hessian matrix of $\mathbf{f}$ with respect to $\mathbf{x}$, and the
initial point $\mathbf{y}_{0}$ is assumed to be a
$\mathcal{F}_{t_{0}}$-measurable random vector.
\end{definition}

Denote by $\mathcal{C}_{P}^{l}$ the space of $l$ time continuously
differentiable functions with partial derivatives up to order $l$
having polynomial growth.

\begin{theorem}
\label{TheoremConvAppLL}Let $\mathbf{x}$ be the solution of the SDE
(\ref{WSDE-LLA-1})-(\ref{WSDE-LLA-2}), and $\mathbf{y}$ the
order-$\beta$ weak
Local Linear discretization of $\mathbf{x}$ defined by (\ref{LLDiscretization}%
). Suppose that the drift and diffusion coefficients of the SDE
(\ref{WSDE-LLA-1}) satisfy the following conditions%
\begin{equation}
\mathbf{f}^{k}\in\mathcal{C}_{P}^{2(\beta+1)}([t_{0},T]\times\mathbb{R}%
^{d},\mathbb{R})\text{ \ \ \ and \ \ \ \ }\mathbf{g}_{i}^{k}\in\mathcal{C}%
_{P}^{2(\beta+1)}([t_{0},T],\mathbb{R})\text{ }\label{W ComponentsCond}%
\end{equation}%
\begin{equation}
\left\vert \mathbf{f}(s,\mathbf{u})\right\vert +%
%TCIMACRO{\dsum \limits_{i=1}^{m}}%
%BeginExpansion
{\displaystyle\sum\limits_{i=1}^{m}}
%EndExpansion
\left\vert \mathbf{g}_{i}(s)\right\vert \leq K(1+\left\vert \mathbf{u}%
\right\vert ),\label{W GrowthBoundLemma}%
\end{equation}
and%
\begin{equation}
\left\vert \frac{\partial\mathbf{f}(s,\mathbf{u})}{\partial
t}\right\vert
+\left\vert \frac{\partial\mathbf{f}(s,\mathbf{u})}{\partial\mathbf{x}%
}\right\vert +\left\vert \frac{\partial^{2}\mathbf{f}(s,\mathbf{u})}%
{\partial\mathbf{x}^{2}}\right\vert \delta_{\beta}^{2}\leq
K\label{W BoundLemma}%
\end{equation}
for all $s\in\lbrack t_{0},T]$ and $\mathbf{u}\in\mathbb{R}^{d}$,
where $K$ is a positive constant. Further, suppose that the initial
values of $\mathbf{x}$ and $\mathbf{y}$ satisfy
\begin{equation}
E(\left\vert \mathbf{x}_{0}\right\vert ^{j})<\infty\text{ },\text{\ \ \ \ }%
E(\left\vert \mathbf{y}_{0}\right\vert ^{j})<\infty,\text{ \ \ \ \
}\left\vert E(g(\mathbf{x}_{0}))-E(g(\mathbf{y}_{0}))\right\vert
\leq C_{0}h^{\beta
}\label{WSDE-LLA-7}%
\end{equation}
for $j=1,2,\ldots,$ some constant $C_{0}>0$ and all $g\in\mathcal{C}%
_{P}^{2(\beta+1)}(\mathbb{R}^{d},\mathbb{R})$. Then there exits a
positive
constant $C_{g}$ such that%
\begin{equation}
\left\vert E\left(  g(\mathbf{x}(T))\right)  -E\left(  g(\mathbf{y}_{n_{T}%
})\right)  \right\vert \leq C_{g}(T-t_{0})h^{\beta}.\label{GlobalOrder}%
\end{equation}

\end{theorem}

By construction, the value $\mathbf{y}_{n+1}$ of the Local Linear
discretization (\ref{LLDiscretization}) is the weak solution of the
piecewise linear SDE
\begin{align}
d\mathbf{z}\left(  t\right)   & =(\mathbf{A}_{n}\mathbf{z}(t)+\mathbf{a}%
_{n}^{\mathbb{\beta}}\left(  t\right)  )dt+\sum\limits_{i=1}^{m}\mathbf{g}%
_{i}(t)d\mathbf{w}^{i}(t)\text{, \ \ }t\in(t_{n},t_{n+1}%
],\label{WSDE LLEquation}\\
\mathbf{z}(t_{n})  & =\mathbf{y}_{n}\text{ \ \ }\label{WSDE LLEquation b}%
\end{align}
at $t_{n+1}$ for all $t_{n+1}\in$ $\left(  t\right)  _{h}$, where
$\mathbf{A}_{n}=\mathbf{f}_{\mathbf{x}}(t_{n},\mathbf{y}_{n})$.

Hereafter, the following definitions and notations from
\cite{Kloeden 1995} are required. Let $\mathcal{M}$ be the set of
all the multi-indexes $\alpha=(j_{1},\ldots,j_{l(\alpha)})$ with
$j_{i}\in\{0,1,\ldots,m\}$ and $i=1,\ldots,l(\alpha)$, where $m$ is
the dimension of $\mathbf{w}$ in (\ref{WSDE-LLA-1}). $l(\alpha)$
denotes the length of the multi-index $\alpha$ and $n(\alpha)$ the
number of its zero components. $-\alpha$ and $\alpha-$ are the
multi-indexes in $\mathcal{M}$ obtained by deleting the first and
the last component of $\alpha$, respectively. The multi-index of
length zero will be
denoted by $v$. Denote by $I_{\alpha}\left[  \cdot\right]  _{t_{n},t_{n}%
+h_{n}}$ the multiple Ito integrals for all
$\mathbb{\alpha}\in\mathcal{M}$. $\ $Further,
\[
L^{0}=\frac{\partial}{\partial t}+\sum\limits_{k=1}^{d}\mathbf{f}^{k}%
\frac{\partial}{\partial\mathbf{x}^{k}}+\frac{1}{2}\sum\limits_{k,l=1}^{d}%
\sum\limits_{j=1}^{m}\mathbf{g}_{j}^{k}\mathbf{g}_{j}^{l}\text{ }%
\frac{\partial^{2}}{\partial\mathbf{x}^{k}\partial\mathbf{x}^{l}}%
\]
denotes the diffusion operator for the SDE (\ref{WSDE-LLA-1}), and
\[
L^{j}=\sum\limits_{k=1}^{d}\mathbf{g}_{j}^{k}\frac{\partial}{\partial
\mathbf{x}^{k}},
\]
for $j=1,\ldots,m$.

\begin{lemma}
\label{LemmaLL} With $\beta=1,2$, let
\[
\Gamma_{\beta}=\left\{
\alpha\in\mathcal{M}:l(\alpha)\leq\beta\right\}
\]
be a hierarchical set, and $\mathcal{B}(\Gamma_{\beta})=\{\mathbb{\alpha}%
\in\mathcal{M}\backslash\Gamma_{\beta}:-\mathbb{\alpha}\in\Gamma_{\beta}\}$
the remainder set of $\Gamma_{\beta}$. Further, let $\mathbf{y}=\{\mathbf{y}%
(t),$ $t\in\lbrack t_{0},T]\}$ be the stochastic process defined as
\begin{equation}
\mathbf{y}(t)=\mathbf{y}_{n_{t}}+\mathbf{\phi}_{\beta}(t_{n_{t}}%
,\mathbf{y}_{n_{t}};t-t_{n_{t}})+\mathbf{\eta}(t_{n_{t}},\mathbf{y}_{n_{t}%
};t-t_{n_{t}}),\label{LLApproximation}%
\end{equation}
where the sequence $\{\mathbf{y}_{n_{t}}\}$, $n=0,1,\ldots,$ is the
Local
Linear discretization (\ref{LLDiscretization}), and let $\mathbf{z}%
=\{\mathbf{z}(t),$ $t\in\lbrack t_{0},T]\}$ be the stochastic
process defined
by%
\begin{equation}
\mathbf{z}(t)=\mathbf{y}_{n_{t}}+\sum\limits_{\alpha\in\Gamma_{\beta}/\{\nu
\}}I_{\alpha}[\Lambda_{\alpha}(t_{n_{t}},\mathbf{y}_{n_{t}};t_{n_{t}%
},\mathbf{y}_{t_{n_{t}}})]_{t_{n_{t}},t}+\sum\limits_{\alpha\in\mathcal{B}%
(\Gamma_{\beta})}I_{\alpha}[\Lambda_{\alpha}(.,\mathbf{y}.;t_{n_{t}%
},\mathbf{y}_{n_{t}})]_{t_{n_{t}},t},\label{WSDE-LLA-3}%
\end{equation}
where, for all given $(t_{n_{t}},\mathbf{y}_{n_{t}})$,%
\[
\Lambda_{\mathbb{\alpha}}(s,\mathbf{v};t_{n_{t}},\mathbf{y}_{n_{t}})=\left\{
\begin{array}
[c]{cc}%
L^{j_{1}}\ldots L^{j_{l(\alpha)-1}}\mathbf{p}_{\beta}(s,\mathbf{v};t_{n_{t}%
},\mathbf{y}_{n_{t}}) & \text{ }if\text{ }j_{l(\alpha)}=0\\
L^{j_{1}}\ldots
L^{j_{l(\alpha)-1}}\mathbf{g}_{j_{l(\mathbb{\alpha)}}}(s) & \text{
}if\text{ }j_{l(\alpha)}\neq0
\end{array}
\right.
\]
is a function of $s$ and $\mathbf{v}$, and%
\[
\mathbf{p}_{\beta}(s,\mathbf{v};r,\mathbf{u})=\left\{
\begin{array}
[c]{cc}%
\mathbf{f}(r,\mathbf{u})+\mathbf{f}_{\mathbf{x}}(r,\mathbf{u})(\mathbf{v-u)}%
+\mathbf{f}_{t}(r,\mathbf{u})(s-r) & \text{for }\mathbb{\beta}=1\\
\mathbf{p}_{1}(s,\mathbf{v};r,\mathbf{u})\text{\ }+\frac{1}{2}\sum
\limits_{j=1}^{m}(\mathbf{I}_{d\times
d}\otimes\mathbf{g}_{j}^{\intercal
}\left(  r\right)  )\mathbf{f}_{\mathbf{xx}}(r,\mathbf{u})\mathbf{g}%
_{j}\left(  r\right)  \left(  s-r\right)  & \text{for
}\mathbb{\beta}=2
\end{array}
\right.
\]
for all $r,s\in\lbrack t_{0},T]$, and $\mathbf{u,v\in}\mathbb{R}^{d}$. Then%
\[
E\left(  g(\mathbf{y}(t))\right)  =E\left(  g(\mathbf{z}(t))\right)
,
\]%
\[
E\left(  g(\mathbf{y}(t)-\mathbf{y}(t_{n_{t}}))\right)  =E\left(
g(\mathbf{z}(t)-\mathbf{z}(t_{n_{t}}))\right)
\]
for all $t\in\lbrack t_{0},T]$ and $g\in\mathcal{C}_{P}^{2(\beta
+1)}(\mathbb{R}^{d},\mathbb{R})$; and%
\begin{equation}
I_{\alpha}[\Lambda_{\alpha}(t_{n_{t}},\mathbf{y}_{n_{t}};t_{n_{t}}%
,\mathbf{y}_{n_{t}})]_{t_{n_{t}},t}=I_{\alpha}[\lambda_{\alpha}(t_{n_{t}%
},\mathbf{y}_{n_{t}})]_{t_{n_{t}},t},\label{IdentityIntegrals}%
\end{equation}
for all $\alpha\in\Gamma_{\beta}/\{\nu\}$ and $t\in\lbrack
t_{0},T]$, where $\lambda_{\alpha}$ denotes the Ito coefficient
function corresponding to the SDE (\ref{WSDE-LLA-1}).
\end{lemma}

Note that, the stochastic process $\mathbf{z}$ defined in the
previous lemma
is the solution of the piecewise linear SDE (\ref{WSDE LLEquation}%
)-(\ref{WSDE LLEquation b}) and $\Lambda_{\mathbb{\alpha}}(\cdot;t_{n_{t}%
},\mathbf{y}_{n_{t}})$ denotes the Ito coefficient functions
corresponding to that equation. Therefore, (\ref{WSDE-LLA-3}) is the
Ito-Taylor expansion of the process (\ref{LLApproximation}), which
coincides with the Local Linear
discretization (\ref{LLDiscretization}) at each discretization time $t_{n}%
\in(t)_{h}$.

\section{Weak Local Linearization schemes}

It can be noted from its definition that the Local Linear
discretization is still no tractable for numerical implementation
purposes. The reason is that, in general, the integrals appearing in
$\mathbf{\phi}_{\beta}$ and $\mathbf{\eta}$ can not be analytically
computed. Thus, depending on the way of computing these functions,
different numerical schemes could be obtained. A precise definition
for such schemes is the following.

\begin{definition}
For an weak Local Linear discretization $\mathbf{y}_{n+1}=\mathbf{y}%
_{n}+\mathbf{\phi}_{\beta}(t_{n},\mathbf{y}_{n};h_{n})+\mathbf{\eta}%
(t_{n},\mathbf{y}_{n};h_{n})$ of the SDE (\ref{WSDE-LLA-1})-(\ref{WSDE-LLA-2}%
), all recursion of the form
\begin{equation}
\widetilde{\mathbf{y}}_{n+1}=\widetilde{\mathbf{y}}_{n}+\widetilde
{\mathbf{\phi}}_{\beta}(t_{n},\widetilde{\mathbf{y}}_{n};h_{n})+\widetilde
{\mathbf{\eta}}(t_{n},\widetilde{\mathbf{y}}_{n};h_{n}),\text{ \ \ \
\ \ with
}\widetilde{\mathbf{y}}_{0}=\mathbf{y}_{0},\label{GWLLS}%
\end{equation}
is called weak Local Linearization scheme, where $\widetilde{\mathbf{\phi}%
}_{\beta}$ and $\widetilde{\mathbf{\eta}}$ denote numerical
algorithms to compute $\mathbf{\phi}_{\beta}$ and $\mathbf{\eta}$
respectively.
\end{definition}

The first weak LL schemes were derived for autonomous equations SDE
(i.e. $\frac{\partial\mathbf{f}}{\partial
t}=\frac{\partial\mathbf{g}_{i}}{\partial t}\equiv0$). By
integrating by part in (\ref{DeterministicLL}), the LL
discretization (\ref{LLDiscretization}) can be rewritten as
\begin{align}
\mathbf{y}_{n+1} &  =\mathbf{y}_{n}+\left(  \mathbf{f}_{\mathbf{x}}%
(\mathbf{y}_{n})\right)  ^{-1}(e^{\mathbf{f}_{\mathbf{x}}(\mathbf{y}_{n}%
)h_{n}}-\mathbf{I}_{d})\mathbf{f}(\mathbf{y}_{n})\label{ShojiOzakiScheme}\\
&  +\frac{\delta_{\beta}^{2}}{2}\left(  \mathbf{f}_{\mathbf{x}}(\mathbf{y}%
_{n})\right)  ^{-2}(e^{\mathbf{f}_{\mathbf{x}}(\mathbf{y}_{n})h_{n}%
}-\mathbf{I}_{d}-\mathbf{f}_{\mathbf{x}}(\mathbf{y}_{n})h_{n})\sum
\limits_{j=1}^{m}(\mathbf{I}_{d\times
d}\otimes\mathbf{g}_{j}^{\intercal
}\left(  t_{n}\right)  )\mathbf{f}_{\mathbf{xx}}(t_{n},\mathbf{y}%
_{n})\mathbf{g}_{j}\left(  t_{n}\right)  h_{n}\nonumber\\
&  +\mathbf{\eta}(t_{n},\mathbf{y}_{n};h_{n}),\nonumber
\end{align}
which for $\beta=1$ and $\beta=2$ were proposed in \cite{Ozaki
1985a, Ozaki 1992} and \cite{Shoji 1997, Shoji 1998c}, respectively.
For scalar SDEs with
$m$ constant diffusion coefficients $\mathbf{g=[}g_{1},\ldots,g_{m}]\in%
%TCIMACRO{\U{211d} }%
%BeginExpansion
\mathbb{R}
%EndExpansion
^{m}$, the variance of
$\mathbf{\eta}(t_{n},\mathbf{y}_{t_{n}};h_{n})$ is given by
\[
\mathbf{\Sigma}(t_{n},\mathbf{y}_{t_{n}};h_{n})=\left(  2\mathbf{f}%
_{\mathbf{x}}(\mathbf{y}_{n})\right)  ^{-1}(e^{2\mathbf{f}_{\mathbf{x}%
}(\mathbf{y}_{n})h_{n}}-1\mathbf{)gg}^{\intercal}\mathbf{,}%
\]
which is obtained by integrating by part in (\ref{CovarianceMatrix})
\cite{Ozaki 1992, Shoji 1997}. For multidimensional autonomous SDEs
the variance $\mathbf{\Sigma}$ of
$\mathbf{\eta}(t_{n},\mathbf{y}_{t_{n}};h_{n})$ is approximated by
the numerical solution of the pencil equation \cite{Shoji
1998c}%
\begin{equation}
\mathbf{f}_{\mathbf{x}}(\mathbf{y}_{n})\mathbf{\Sigma+\Sigma f}_{\mathbf{x}%
}^{\intercal}(\mathbf{y}_{n})=e^{\mathbf{f}_{\mathbf{x}}(\mathbf{y}_{n})h_{n}%
}\mathbf{GG}^{\intercal}e^{\left(  \mathbf{f}_{\mathbf{x}}(\mathbf{y}%
_{n})\right)  ^{\top}h_{n}}-\mathbf{GG}^{\intercal},\label{PencilEquation}%
\end{equation}
where $\mathbf{G=[g}_{1},\ldots,\mathbf{g}_{m}]$ is an $d\times m$
matrix of constant entries. However, the numerical implementations
$\widetilde {\mathbf{y}}_{n+1}$ of the LL discretization
(\ref{ShojiOzakiScheme}) (i.e., the corresponding LL schemes) are
not always computationally feasible since
they might eventually fail when the Jacobian matrix $\mathbf{f}_{\mathbf{x}%
}(\widetilde{\mathbf{y}}_{n})$ is singular or ill-conditioned at
some point $\widetilde{\mathbf{y}}_{n}$. Moreover, the equation
(\ref{PencilEquation})
might have no unique solution for some particular $\mathbf{f}_{\mathbf{x}%
}(\widetilde{\mathbf{y}}_{n})$.

For nonautonomous equations, the LL discretization
(\ref{LLDiscretization})
can be written as \cite{Mora05}%
\begin{align*}
\mathbf{y}_{n+1} &  =h_{n}e^{\mathbf{f}_{\mathbf{x}}(t_{n},\mathbf{y}%
_{n})\frac{h_{n}}{2}}(\mathbf{f}(t_{n},\mathbf{y}_{n})-\mathbf{f}_{\mathbf{x}%
}(t_{n},\mathbf{y}_{n})\mathbf{y}_{n}+\frac{h_{n}}{2}\mathbf{f}_{t}%
(t_{n},\mathbf{y}_{n})\\
&  +\frac{h_{n}}{4}\sum\limits_{j=1}^{m}(\mathbf{I}_{d\times
d}\otimes
\mathbf{g}_{j}^{\intercal}\left(  t_{n}\right)  )\mathbf{f}_{\mathbf{xx}%
}(t_{n},\mathbf{y}_{n})\mathbf{g}_{j}\left(  t_{n}\right)  )\\
&  +e^{\mathbf{f}_{\mathbf{x}}(t_{n},\mathbf{y}_{n})h_{n}}\mathbf{y}%
_{n}+\widetilde{\mathbf{\Sigma}}^{1/2}(t_{n},\mathbf{y}_{n};h_{n})\xi
_{n+1}+\mathbf{r}(t_{n+1}),
\end{align*}
where the variance $\mathbf{\Sigma}$ of
$\mathbf{\eta}(t_{n},\widetilde
{\mathbf{y}}_{t_{n}};h_{n})$ is approximated by%
\[
\widetilde{\mathbf{\Sigma}}(t_{n},\mathbf{y}_{n};h_{n})=h_{n}e^{\mathbf{f}%
_{\mathbf{x}}(t_{n},\mathbf{y}_{n})\frac{h_{n}}{2}}\mathbf{G(}t_{n}%
+\frac{h_{n}}{2}\mathbf{)G(}t_{n}+\frac{h_{n}}{2}\mathbf{)}^{\intercal
}e^{\mathbf{f}_{\mathbf{x}}^{\intercal}(t_{n},\mathbf{y}_{n})\frac{h_{n}}{2}},
\]
$\{\xi_{n}\}$ is a sequence of $d$-dimensional i.i.d Gaussian random
vectors, and $\mathbf{r}$ is a remainder term. These expressions can
be obtained after some algebraic manipulations in
(\ref{LLDiscretization}) and by using quadrature formulas for
approximating the integrals (\ref{DeterministicLL})
and (\ref{CovarianceMatrix}) with $\beta=2$. The remainder term $\mathbf{r}%
(t_{n+1})$ represents the error due to these approximations. The LL
schemes $\widetilde{\mathbf{y}}_{n+1}$ that can be obtained by
numerical implementations of the above expression for
$\mathbf{y}_{n+1}$ (and neglecting $\mathbf{r}(t_{n+1})$) overcome
the restrictions for the Jacobian matrix $\mathbf{f}_{\mathbf{x}}$
of the previous ones, but at expense of an additional approximation.

Alternatively, other types of weak LL schemes have be proposed
\cite{Carbonell06}.

For SDEs with constant diffusion coefficients, i.e., for equations
for the form (\ref{WSDE-LLA-1}) with
$\mathbf{g}_{i}(t)\equiv\mathbf{g}_{i}$ for all $t\in\lbrack
t_{0},T]$, Theorem 1 in \cite{Van Loan 1978} implies that
$\mathbf{\phi}_{\beta}$ and $\mathbf{\Sigma}$ can be rewritten as%
\begin{align*}
\mathbf{\phi}_{\beta}(t_{n},\mathbf{y}_{n};h_{n})  & =\int\limits_{0}^{h_{n}%
}e^{\mathbf{f}_{\mathbf{x}}(t_{n},\mathbf{y}_{n})(h_{n}-s)}\mathbf{f}%
(t_{n},\mathbf{y}_{n})ds+\int\limits_{0}^{h_{n}}\int\limits_{0}^{s}%
e^{\mathbf{f}_{\mathbf{x}}(t_{n},\mathbf{y}_{n})(h_{n}-s)}\mathbf{b}_{\beta
}(t_{n},\mathbf{y}_{n})ds_{1}ds\\
& =\mathbf{D}_{14}(t_{n},\mathbf{y}_{n};h_{n}),\\
\mathbf{\Sigma}(t_{n},\mathbf{y}_{n};h_{n})  & =(\int\limits_{0}^{h_{n}%
}e^{\mathbf{f}_{\mathbf{x}}(t_{n},\mathbf{y}_{n})(h_{n}-s)}\mathbf{GG}%
^{\intercal}e^{-\mathbf{f}_{\mathbf{x}}(t_{n},\mathbf{y}_{n})^{\intercal}%
s}ds)\text{ }e^{\mathbf{f}_{\mathbf{x}}(t_{n},\mathbf{y}_{n})^{\intercal}%
h_{n}}\\
&
=\mathbf{D}_{12}(t_{n},\mathbf{y}_{n};h_{n})\mathbf{D}_{11}^{\intercal
}(t_{n},\mathbf{y}_{n};h_{n}),
\end{align*}
where
\[
\mathbf{b}_{\beta}(t_{n},\mathbf{y}_{n})=\mathbf{a}_{\beta}(t_{n}%
,\mathbf{y}_{n})-\mathbf{f}_{\mathbf{x}}(t_{n},\mathbf{y}_{n})\mathbf{y}_{n},
\]
$\mathbf{a}_{\beta}$ is the function defined in
(\ref{LLDiscretization}), $\mathbf{G=[g}_{1},\ldots,\mathbf{g}_{m}]$
is an $d\times m$ matrix, and the block matrix
$\mathbf{D}=(\mathbf{D}_{lj}),$ $l,j=1,\ldots,4$ is defined as
$\mathbf{D(}t_{n}\mathbf{,y}_{n}\mathbf{;}h_{n}\mathbf{)}=e^{\mathbf{C}%
_{\beta}(t_{n},\mathbf{y}_{n})h_{n}}$\textbf{,} with
\[
\mathbf{C}_{\beta}(t_{n},\mathbf{y}_{n})=%
\begin{pmatrix}
\mathbf{f}_{\mathbf{x}}(t_{n},\mathbf{y}_{n}) &
\mathbf{GG}^{\intercal} &
\mathbf{b}_{\beta}(t_{n},\mathbf{y}_{n}) & \mathbf{f}(t_{n},\mathbf{y}_{n})\\
\mathbf{0} &
-\mathbf{f}_{\mathbf{x}}(t_{n},\mathbf{y}_{n})^{\intercal} &
\mathbf{0} & \mathbf{0}\\
\mathbf{0} & \mathbf{0} & 0 & 1\\
\mathbf{0} & \mathbf{0} & 0 & 0
\end{pmatrix}
.
\]
Therefore, the LL discretization (\ref{LLDiscretization}) can be
written as
\begin{equation}
\mathbf{y}_{n+1}=\mathbf{y}_{n}+\mathbf{D}_{14}(t_{n},\mathbf{y}_{n}%
;h_{n})+\mathbf{(D}_{12}\mathbf{(}t_{n},\mathbf{y}_{n};h_{n}\mathbf{)D}%
_{11}^{\intercal}\mathbf{(}t_{n},\mathbf{y}_{n};h_{n}\mathbf{))}^{1/2}%
\xi_{n+1},\label{Scheme1}%
\end{equation}
starting with $\mathbf{y}_{0}=\mathbf{x}_{0}.$ Here\textbf{,
}$\mathbf{\Sigma }^{1/2}$ denotes the square root matrix of
$\mathbf{\Sigma}$ and $\{\xi_{n}\}$ is a sequence of $d$-dimensional
i.i.d Gaussian random vectors.

In general, for SDEs of the form (\ref{WSDE-LLA-1}) with no constant
diffusion
coefficients\textbf{,} the approximation%
\[
\mathbf{G}(t_{n}+s)\approx\mathbf{G}_{\beta}(t_{n}+s),\text{ }s\in
\lbrack0,h_{n}]
\]
provided by the truncated Taylor expansion
\[
\mathbf{G}_{\beta}(t_{n}+s)=\sum\limits_{i=0}^{\beta-1}\frac{d^{i}%
\mathbf{G}(t_{n})}{dt^{i}}s^{i}%
\]
has been considered. In turn, this implies that%
\[
\mathbf{\Sigma}(t_{n},\mathbf{y}_{n};h_{n})\approx\mathbf{\Sigma}_{\beta
}(t_{n},\mathbf{y}_{n};h_{n}),
\]
where%
\begin{align*}
&  \mathbf{\Sigma}_{\beta}(t_{n},\mathbf{y}_{n};h_{n})\\
&  =(\int\limits_{0}^{h_{n}}e^{\mathbf{f}_{\mathbf{x}}(t_{n},\mathbf{y}%
_{n})(h_{n}-s)}\mathbf{H}_{0}\mathbf{(}t_{n}\mathbf{)}e^{-\mathbf{f}%
_{\mathbf{x}}(t_{n},\mathbf{y}_{n})^{\intercal}s}ds\\
&  +\sum\limits_{i=1}^{2\beta-2}\int\limits_{0}^{h_{n}}\int\limits_{0}^{s}%
\int\limits_{0}^{s_{0}}\ldots\int\limits_{0}^{s_{i-2}}e^{\mathbf{f}%
_{\mathbf{x}}(t_{n},\mathbf{y}_{n})(h_{n}-s)}\mathbf{H}_{i}(t_{n}%
)e^{-\mathbf{f}_{\mathbf{x}}(t_{n},\mathbf{y}_{n})^{\intercal}s}ds_{i-1}\ldots
ds_{0}ds)e^{\mathbf{f}_{\mathbf{x}}(t_{n},\mathbf{y}_{n})^{\intercal}h_{n}},
\end{align*}
with%
\[
\mathbf{H}_{i}(t_{n})=\sum\limits_{l+j=i}\frac{d^{l}\mathbf{G}(t_{n})}{dt^{l}%
}\frac{d^{j}\mathbf{G}(t_{n})}{dt^{j}}^{\intercal},\text{
}i=0,\ldots ,2\beta-2.
\]
Hence, by Theorem 1 in \cite{Carbonell08 JCAM} it is obtained%
\begin{align}
\mathbf{\phi}_{\beta}(t_{n},\mathbf{y}_{n};h_{n})  &
=\mathbf{B}_{1,2\beta
+2}(t_{n},\mathbf{y}_{n};h_{n}),\nonumber\\
\mathbf{\Sigma}_{\beta}(t_{n},\mathbf{y}_{n};h_{n})  &
=\mathbf{B}_{1,2\beta
}(t_{n},\mathbf{y}_{n};h_{n})\mathbf{B}_{11}^{\intercal}(t_{n},\mathbf{y}%
_{n};h_{n}),\label{CovarianceMatrixBeta}%
\end{align}
where the block matrix $\mathbf{B}=(\mathbf{B}_{lj})$ is defined as
\begin{equation}
\mathbf{B}(t_{n},\mathbf{y}_{n};h_{n})=e^{\mathbf{A}_{\beta}(t_{n}%
,\mathbf{y}_{n})h_{n}}\label{matrix B WLL scheme}%
\end{equation}
with%
\begin{align*}
& \mathbf{A}_{\beta}(t_{n},\mathbf{y}_{n})\\
& =%
\begin{pmatrix}
\mathbf{f}_{\mathbf{x}}(t_{n},\mathbf{y}_{n}) &
\mathbf{H}_{2\beta-2}(t_{n}) &
\mathbf{H}_{2\beta-3}(t_{n}) & \cdots & \mathbf{H}_{0}(t_{n}) & \mathbf{b}%
_{\beta}(t_{n},\mathbf{y}_{n}) & \mathbf{f}(t_{n},\mathbf{y}_{n})\\
\mathbf{0} &
-\mathbf{f}_{\mathbf{x}}(t_{n},\mathbf{y}_{n})^{\intercal} &
\mathbf{I}_{d} & \cdots & \mathbf{0} & \mathbf{0} & \mathbf{0}\\
\vdots & \vdots &
-\mathbf{f}_{\mathbf{x}}(t_{n},\mathbf{y}_{n})^{\intercal} &
\ddots & \vdots & \vdots & \vdots\\
\mathbf{0} & \mathbf{0} & \mathbf{0} & \ddots & \mathbf{I}_{d} &
\vdots &
\vdots\\
\vdots & \vdots & \vdots & \ddots & -\mathbf{f}_{\mathbf{x}}(t_{n}%
,\mathbf{y}_{n})^{\intercal} & \mathbf{0} & \mathbf{0}\\
\vdots & \vdots & \vdots & \cdots & \mathbf{0} & 0 & 1\\
\mathbf{0} & \mathbf{0} & \mathbf{0} & \mathbf{\cdots} & \mathbf{0}
& 0 & 0
\end{pmatrix}
.
\end{align*}
In this way, the LL discretization (\ref{LLDiscretization}) can be
written as
\begin{align}
\mathbf{y}_{n+1}  & =\mathbf{y}_{n}+\mathbf{B}_{1,2\beta+2}(t_{n}%
,\mathbf{y}_{n};h_{n})\label{Scheme2}\\
& +(\mathbf{B}_{1,2\beta}(t_{n},\mathbf{y}_{n};h_{n})\mathbf{B}_{1,1}%
^{\intercal}(t_{n},\mathbf{y}_{n};h_{n}))^{1/2}\xi_{n+1}\nonumber\\
& +\mathbf{r}(t_{n},\mathbf{y}_{n};h_{n}),\nonumber
\end{align}
where $\mathbf{r}(t_{n},\mathbf{y}_{n};h_{n})=\mathbf{\Sigma}^{1/2}%
(t_{n},\mathbf{y}_{n};h_{n})\xi_{n+1}-\mathbf{\Sigma}_{\beta}^{1/2}%
(t_{n},\mathbf{y}_{n};h_{n})\xi_{n+1}$.

Note that numerical implementations $\widetilde{\mathbf{y}}_{n+1}$
of the expressions (\ref{Scheme1}) and (\ref{Scheme2}) reduce to the
use of a suitable algorithm for computing exponential matrices, for
intance, those based on rational Pad\'{e} approximations or Krylov
subspace method (see \cite{Moler 2003} for an updated review).
Remarkably, these expressions have no restriction on the Jacobian
matrix $\mathbf{f}_{\mathbf{x}}$, and do not involve the use of
quadrature formulas either.

Furthermore, for weak convergence purpose, the sequence
$\{\xi_{n}\}$ of i.i.d Gaussian random vectors can be replaced by
any other sequence of random vectors with similar moment properties.
Thus, in all numerical implementation of the LL discretization,
$\{\xi_{n}\}$ can be replaced by the sequence
$\{\widetilde{\xi}_{n}\}$ of i.i.d. two-points distributed random
vectors with
components $\widetilde{\xi}_{n}^{k}$ satisfying $P(\widetilde{\xi}_{n}^{k}%
=\pm1)=1/2$.

\section{Convergence rate of the Weak Local Linearization schemes}

Clearly, a weak LL scheme will preserve the order $\beta$ of the
underlaying LL discretization if
$\widetilde{\mathbf{\phi}}\mathbf{_{\mathbb{\beta}}}$ and
$\widetilde{\mathbf{\Sigma}}$ are suitable approximations to
$\mathbf{\phi }_{\beta}$ and $\mathbf{\Sigma}$. This requirement is
considered in the following results.

\begin{lemma}
\label{BWLLS}Suppose that the drift and diffusion coefficients of
the SDE (\ref{WSDE-LLA-1})-(\ref{WSDE-LLA-2}) satisfy the conditions
(\ref{W GrowthBoundLemma})-(\ref{W BoundLemma}). Further suppose that%
\[
\left\vert \mathbf{\phi_{\beta}}\left(  t_{n},\widetilde{\mathbf{y}}_{n}%
;h_{n}\right)
-\widetilde{\mathbf{\phi}}\mathbf{_{\mathbb{\beta}}}\left(
t_{n},\widetilde{\mathbf{y}}_{n};h_{n}\right)  \right\vert \leq
K(1+\left\vert
\widetilde{\mathbf{y}}_{n}\right\vert )h_{n}^{\alpha+1}%
\]
and
\[
\left\vert \mathbf{\Sigma}(t_{n},\widetilde{\mathbf{y}}_{n};h_{n}%
)-\widetilde{\mathbf{\Sigma}}(t_{n},\widetilde{\mathbf{y}}_{n};h_{n}%
)\right\vert \leq K(1+\left\vert
\widetilde{\mathbf{y}}_{n}\right\vert )h_{n}^{\gamma+1},
\]
for some positive constant $K$ and natural numbers $\alpha$ and
$\gamma$. If the initial value $\widetilde{\mathbf{y}}_{0}$ of a LL
scheme \thinspace $\widetilde{\mathbf{y}}_{n}$ defined as in
(\ref{GWLLS}) has finite moments of all orders, then
\[
E\left(  \underset{0\leq n\leq n_{T}}{\max}\left\vert \widetilde{\mathbf{y}%
}_{n}\right\vert ^{2j}{\LARGE |}\mathcal{F}_{t_{0}}\right)  \leq
M(1+\left\vert \widetilde{\mathbf{y}}_{0}\right\vert ^{2r})
\]
for some positive constant $M$ and natural numbers $j$ and $r$.
\end{lemma}

\begin{proof}
By using conditions (\ref{W GrowthBoundLemma})-(\ref{W BoundLemma})
it is
obtained that%
\[
\left\vert \mathbf{\phi}_{\beta}(t_{n},\widetilde{\mathbf{y}}_{n}%
;h_{n})\right\vert \leq C_{1}(1+\left\vert \widetilde{\mathbf{y}}%
_{n}\right\vert )h_{n}%
\]
and%
\[
\left\vert
\mathbf{\Sigma}(t_{n},\widetilde{\mathbf{y}}_{n};h_{n})\right\vert
\leq C_{2}(1+\left\vert \widetilde{\mathbf{y}}_{n}\right\vert )^{2}h_{n}%
\]
where $C_{1}$ and $C_{2}$\ are positive constants. In this way,%
\begin{align}
\left\vert \widetilde{\mathbf{\phi}}_{\beta}(t_{n},\widetilde{\mathbf{y}}%
_{n};h_{n})\right\vert  & \leq\left\vert \mathbf{\phi}_{\beta}(t_{n}%
,\widetilde{\mathbf{y}}_{n};h_{n})\right\vert +\left\vert \widetilde
{\mathbf{\phi}}_{\beta}(t_{n},\widetilde{\mathbf{y}}_{n};h_{n})-\mathbf{\phi
}_{\beta}(t_{n},\widetilde{\mathbf{y}}_{n};h_{n})\right\vert \nonumber\\
& \leq K_{1}(1+\left\vert \widetilde{\mathbf{y}}_{n}\right\vert )h_{n}%
\label{WSDE-S5}%
\end{align}
and%
\begin{align}
\left\vert \widetilde{\mathbf{\Sigma}}(t_{n},\widetilde{\mathbf{y}}_{n}%
;h_{n})\right\vert  & \leq\left\vert
\mathbf{\Sigma}(t_{n},\widetilde
{\mathbf{y}}_{n};h_{n})\right\vert +\left\vert \widetilde{\mathbf{\Sigma}%
}(t_{n},\widetilde{\mathbf{y}}_{n};h_{n})-\mathbf{\Sigma}(t_{n},\widetilde
{\mathbf{y}}_{n};h_{n})\right\vert \nonumber\\
& \leq K_{2}(1+\left\vert \widetilde{\mathbf{y}}_{n}\right\vert )^{2}%
h_{n},\label{WSDE-S6}%
\end{align}
where $K_{1}=C_{1}+K$ and $K_{2}=C_{2}+K$.

From (\ref{WSDE-S5}) follows that
\begin{align}
\left\vert E\left(
\widetilde{\mathbf{\phi}}_{\beta}(t_{n},\widetilde
{\mathbf{y}}_{n};h_{n})+\widetilde{\mathbf{\Sigma}}(t_{n},\widetilde
{\mathbf{y}}_{n};h_{n})^{1/2}\xi_{n+1}{\LARGE
|}\mathcal{F}_{t_{n}}\right)
\right\vert  & =\left\vert E\left(  \widetilde{\mathbf{\phi}}_{\beta}%
(t_{n},\widetilde{\mathbf{y}}_{n};h_{n}){\LARGE
|}\mathcal{F}_{t_{n}}\right)
\right\vert \nonumber\\
& \leq K_{1}(1+\left\vert \widetilde{\mathbf{y}}_{n}\right\vert )h_{n}%
,\label{WSDE-S3}%
\end{align}
whereas (\ref{WSDE-S6}) implies that%
\begin{align*}
\left\vert \widetilde{\mathbf{\Sigma}}(t_{n},\widetilde{\mathbf{y}}_{n}%
;h_{n})^{1/2}\xi_{n+1}\right\vert ^{2}  & =\xi_{n+1}^{\intercal}%
\widetilde{\mathbf{\Sigma}}(t_{n},\widetilde{\mathbf{y}}_{n};h_{n})\xi_{n+1}\\
& \leq\left\vert \xi_{n+1}^{\intercal}\right\vert \left\vert
\widetilde
{\mathbf{\Sigma}}(t_{n},\widetilde{\mathbf{y}}_{n};h_{n})\right\vert
\left\vert \xi_{n+1}\right\vert \\
& \leq\left\vert \xi_{n+1}\right\vert ^{2}K_{2}(1+\left\vert
\widetilde
{\mathbf{y}}_{n}\right\vert )^{2}h_{n}%
\end{align*}
and so%
\begin{equation}
\left\vert \widetilde{\mathbf{\phi}}_{\beta}(t_{n},\widetilde{\mathbf{y}}%
_{n};h_{n})+\widetilde{\mathbf{\Sigma}}(t_{n},\widetilde{\mathbf{y}}_{n}%
;h_{n})^{1/2}\xi_{n+1}\right\vert \leq
M(\mathbf{\xi}_{n+1})(1+\left\vert
\widetilde{\mathbf{y}}_{n}\right\vert )h_{n}^{1/2},\label{WSDE-S4}%
\end{equation}
where $M(\mathbf{\xi}_{n+1})=(K_{1}+\sqrt{K_{2}}\left\vert \xi_{n+1}%
\right\vert )$ is a random variable\ with finite moments of all
orders.

Inequalities (\ref{WSDE-S3})-(\ref{WSDE-S4}), condition
$E(\left\vert \mathbf{y}_{0}\right\vert ^{j})<\infty$, for all
$j=1,2,...$, and Lemma 9.1 in
\cite{Milshtein 1988} imply that $E\left(  \left\vert \widetilde{\mathbf{y}%
}_{n}\right\vert ^{2j}\right)  $ exists and is uniformly bounded
with respect to $n_{T}$ for all $n=0,\ldots,n_{T}$, which directly
implies the assertion of the theorem.
\end{proof}

The main convergence result is the following.

\begin{theorem}
\label{WSCT}Let $\mathbf{x}$ be the solution of the SDE (\ref{WSDE-LLA-1}%
)-(\ref{WSDE-LLA-2}),
\[
\mathbf{y}_{n+1}=\mathbf{y}_{n}+\mathbf{\phi}_{\beta}(t_{n},\mathbf{y}%
_{n};h_{n})+\mathbf{\Sigma}(t_{n},\mathbf{y}_{n};h_{n})^{1/2}\xi_{n+1}%
\]
the weak Local Linear discretization defined in
(\ref{LLDiscretization}), and
\[
\widetilde{\mathbf{y}}_{n+1}=\widetilde{\mathbf{y}}_{n}+\widetilde
{\mathbf{\phi}}_{\beta}(t_{n},\widetilde{\mathbf{y}}_{n};h_{n})+\widetilde
{\mathbf{\Sigma}}(t_{n},\widetilde{\mathbf{y}}_{n};h_{n})^{1/2}\xi_{n+1}%
\]
a numerical implementation of $\mathbf{y}_{n+1}$, where $\widetilde
{\mathbf{\phi}}_{\mathbb{\beta}}$ and $\widetilde{\mathbf{\Sigma}}$
denote numerical algorithms for computing $\mathbf{\phi}_{\beta}$
and $\mathbf{\Sigma }$. Suppose that
$\widetilde{\mathbf{\phi}}_{\mathbb{\beta}}$ and
$\widetilde{\mathbf{\Sigma}}$ fulfill the local conditions%
\begin{equation}
\left\vert \mathbf{\phi}_{\beta}(t_{n},\widetilde{\mathbf{y}}_{n}%
;h_{n})-\widetilde{\mathbf{\phi}}_{\beta}(t_{n},\widetilde{\mathbf{y}}%
_{n};h_{n})\right\vert \leq K(1+\left\vert \widetilde{\mathbf{y}}%
_{n}\right\vert )h_{n}^{\alpha+1},\label{WSDE-LLA-9}%
\end{equation}
and%
\begin{equation}
\left\vert \mathbf{\Sigma}(t_{n},\widetilde{\mathbf{y}}_{n};h_{n}%
)-\widetilde{\mathbf{\Sigma}}(t_{n},\widetilde{\mathbf{y}}_{n};h_{n}%
)\right\vert \leq K(1+\left\vert
\widetilde{\mathbf{y}}_{n}\right\vert
)h_{n}^{\gamma+1}.\label{WSDE-LLA-10}%
\end{equation}
Then, under the assumptions of Theorem \ref{TheoremConvAppLL}, there
exits a
positive constant $C_{g}$ such that%
\[
\left\vert E\left(  g(\mathbf{x}(T))\right)  -E\left(  g(\widetilde
{\mathbf{y}}_{n_{T}})\right)  \right\vert \leq C_{g}h^{\min\{\alpha
,\beta,\gamma\}},
\]
for all
$g\in\mathcal{C}_{P}^{2(\beta+1)}(\mathbb{R}^{d},\mathbb{R})$.
\end{theorem}

\begin{proof}
From Lemma \ref{BWLLS} we have%
\[
E\left(  \underset{0\leq n\leq n_{T}}{\max}\left\vert \widetilde{\mathbf{y}%
}_{n}\right\vert ^{2j}{\LARGE |}\mathcal{F}_{t_{0}}\right)  \leq
K_{1}(1+\left\vert \widetilde{\mathbf{y}}_{0}\right\vert ^{2r})
\]
for some positive constant $K_{1}$ and natural numbers $j$ and $r$.

On the other hand, noted that
\begin{equation}
\widetilde{\mathbf{y}}_{n+1}=\widetilde{\mathbf{z}}_{n+1}+\widetilde
{\mathbf{\phi}}_{\beta}(t_{n},\widetilde{\mathbf{y}}_{n};h_{n})-\mathbf{\phi
}_{\beta}(t_{n},\widetilde{\mathbf{y}}_{n};h_{n})+(\widetilde{\mathbf{\Sigma}%
}(t_{n},\widetilde{\mathbf{y}}_{n};h_{n})^{1/2}-\mathbf{\Sigma}(t_{n}%
,\widetilde{\mathbf{y}}_{n};h_{n})^{1/2})\xi_{n+1},\label{YTruncAux}%
\end{equation}
where
\[
\widetilde{\mathbf{z}}_{n+1}=\widetilde{\mathbf{y}}_{n}+\mathbf{\phi}_{\beta
}(t_{n},\widetilde{\mathbf{y}}_{n};h_{n})+\mathbf{\Sigma}(t_{n},\widetilde
{\mathbf{y}}_{n};h_{n})^{1/2}\xi_{n+1}%
\]
is the solution of the linear SDE%
\begin{align}
d\widetilde{\mathbf{z}}(t)  & =\mathbf{p}_{\beta}(t,\widetilde{\mathbf{z}%
}\mathbf{(}t\mathbf{);}t_{n},\widetilde{\mathbf{y}}_{n})dt+\mathbf{G}%
(t)d\mathbf{w}(t)\label{WSDE-LLA-23}\\
\widetilde{\mathbf{z}}(t_{n})  &
=\widetilde{\mathbf{y}}_{n}\nonumber
\end{align}
for all $t\in\lbrack t_{n},t_{n+1}]$ and $t_{n}\in(t)_{h}$, where
the function $\mathbf{p}_{\beta}$ is defined as in Lemma
\ref{LemmaLL}. Thus,
\begin{align*}
\widetilde{\mathbf{y}}_{n+1}-\widetilde{\mathbf{y}}_{n}  &
=\widetilde
{\mathbf{z}}_{n+1}-\widetilde{\mathbf{z}}_{n}+\widetilde{\mathbf{\phi}}%
_{\beta}(t_{n},\widetilde{\mathbf{y}}_{n};h_{n})-\mathbf{\phi}_{\beta}%
(t_{n},\widetilde{\mathbf{y}}_{n};h_{n})+(\widetilde{\mathbf{\Sigma}}%
(t_{n},\widetilde{\mathbf{y}}_{n};h_{n})^{1/2}\\
&
-\mathbf{\Sigma}(t_{n},\widetilde{\mathbf{y}}_{n};h_{n})^{1/2})\xi_{n+1}.
\end{align*}
From this and the algebraic inequality
$(a+b)^{2j}\leq2^{2j-1}(a^{2j}+b^{2j})$ it is obtained that
\begin{align}
E\left(  \left\vert \widetilde{\mathbf{y}}_{n+1}-\widetilde{\mathbf{y}}%
_{n}\right\vert ^{2j}{\LARGE |}\mathcal{F}_{t_{n}}\right)   &  \leq
2^{2j-1}E\left(  \left\vert \widetilde{\mathbf{z}}_{n+1}-\widetilde
{\mathbf{z}}_{n}\right\vert ^{2j}{\LARGE
|}\mathcal{F}_{t_{n}}\right)
\label{WSDE-LLA-20}\\
&  +2^{4j-2}E\left(  \left\vert \widetilde{\mathbf{\phi}}_{\beta}%
(t_{n},\widetilde{\mathbf{y}}_{n};h_{n})-\mathbf{\phi}_{\beta}(t_{n}%
,\widetilde{\mathbf{y}}_{n};h_{n})\right\vert ^{2j}{\LARGE |}\mathcal{F}%
_{t_{n}}\right)  \nonumber\\
&  +2^{4j-2}E\left(  \left\vert (\widetilde{\mathbf{\Sigma}}(t_{n}%
,\widetilde{\mathbf{y}}_{n};h_{n})^{1/2}-\mathbf{\Sigma}(t_{n},\widetilde
{\mathbf{y}}_{n};h_{n})^{1/2})\xi_{n+1}\right\vert ^{2j}{\LARGE |}%
\mathcal{F}_{t_{n}}\right)  \nonumber
\end{align}
By Theorem 4.5.4 in \cite{Kloeden 1995} follows that%
\begin{equation}
E\left(  \left\vert \widetilde{\mathbf{z}}_{n+1}-\widetilde{\mathbf{z}}%
_{n}\right\vert ^{2j}{\LARGE |}\mathcal{F}_{t_{n}}\right)  \leq K_{2}%
(1+\left\vert \widetilde{\mathbf{y}}_{n}\right\vert ^{2j})h_{n}^{j}%
,\label{WSDE-LLA-5}%
\end{equation}
where $K_{2}$ is a positive constant. From condition
(\ref{WSDE-LLA-9}), and
by using that $h_{n}^{2j(\alpha+1/2)}<1$, it is obtained that%
\begin{equation}
E\left(  \left\vert
\widetilde{\mathbf{\phi}}_{\beta}(t_{n},\widetilde
{\mathbf{y}}_{n};h_{n})-\mathbf{\phi}_{\beta}(t_{n},\widetilde{\mathbf{y}}%
_{n};h_{n})\right\vert ^{2j}{\LARGE |}\mathcal{F}_{t_{n}}\right)
\leq
K_{3}(1+\left\vert \widetilde{\mathbf{y}}_{n}\right\vert ^{2j})h_{n}%
^{j},\label{WSDE-LLA-19}%
\end{equation}
where $K_{3}=2^{2j-1}K$. Furthermore, due to the perturbation bounds
for the Cholesky and SVD factorizations (Theorems 2.2.1 and 3.2.1 in
\cite{Sun92})
there exists a positive constant $C$ such that%
\begin{equation}
\left\vert \widetilde{\mathbf{\Sigma}}(t_{n},\widetilde{\mathbf{y}}_{n}%
;h_{n})^{1/2}-\mathbf{\Sigma}(t_{n},\widetilde{\mathbf{y}}_{n};h_{n}%
)^{1/2}\right\vert \leq C\left\vert \widetilde{\mathbf{\Sigma}}(t_{n}%
,\widetilde{\mathbf{y}}_{n};h_{n})-\mathbf{\Sigma}(t_{n},\widetilde
{\mathbf{y}}_{n};h_{n})\right\vert .\label{WSDE-LLA-21}%
\end{equation}
From this, condition (\ref{WSDE-LLA-10}) and taking in to account
that
$h_{n}^{2j(\gamma+1/2)}<1$ it is obtained that%
\begin{align}
&  E\left(  \left\vert ((\mathbf{\Sigma}(t_{n},\widetilde{\mathbf{y}}%
_{n};h_{n})^{1/2}-\widetilde{\mathbf{\Sigma}}(t_{n},\widetilde{\mathbf{y}}%
_{n};h_{n})^{1/2})\xi_{n+1}\right\vert ^{2j}{\LARGE |}\mathcal{F}_{t_{n}%
}\right)  \nonumber\\
&  \leq\left\vert \mathbf{\Sigma}(t_{n},\widetilde{\mathbf{y}}_{n}%
;h_{n})^{1/2}-\widetilde{\mathbf{\Sigma}}(t_{n},\widetilde{\mathbf{y}}%
_{n};h_{n})^{1/2}\right\vert ^{2j}E\left\vert \xi_{n+1}\right\vert
^{2j}\nonumber\\
&  \leq K_{4}(1+\left\vert \widetilde{\mathbf{y}}_{n}\right\vert ^{2j}%
)h_{n}^{j},\label{WSDE-LLA-18}%
\end{align}
where $K_{4}=2^{2j-1}(CK)^{2j}E\left\vert \xi_{n+1}\right\vert
^{2j}$. Inequalities (\ref{WSDE-LLA-20})-(\ref{WSDE-LLA-19}) and
(\ref{WSDE-LLA-18})
yield to%
\begin{equation}
E\left(  \left\vert \widetilde{\mathbf{y}}_{n+1}-\widetilde{\mathbf{y}}%
_{n}\right\vert ^{2j}{\LARGE |}\mathcal{F}_{t_{n}}\right)  \leq K_{5}%
(1+\left\vert \widetilde{\mathbf{y}}_{n}\right\vert ^{2j})h_{n}^{j}%
,\label{WSDE-LLA-6}%
\end{equation}
where $K_{5}$ is a positive constant.

In addition, by the triangular inequality, we have that%
\[
\left\vert E\left(  \mathbf{F}_{\mathbf{p}}(\widetilde{\mathbf{y}}%
_{n+1}-\widetilde{\mathbf{y}}_{n})-\mathbf{F}_{\mathbf{p}}(\sum\limits_{\alpha
\in\Gamma_{\beta}/\{\nu\}}I_{\alpha}[\lambda_{\alpha}(t_{n},\widetilde
{\mathbf{y}}_{n})]_{t_{n},t_{n+1}}){\LARGE
|}\mathcal{F}_{t_{n}}\right) \right\vert \leq e_{1}+e_{2},
\]
where%
\[
e_{1}=\left\vert E\left(  \mathbf{F}_{\mathbf{p}}(\widetilde{\mathbf{z}}%
_{n+1}-\widetilde{\mathbf{y}}_{n})-\mathbf{F}_{\mathbf{p}}(\sum\limits_{\alpha
\in\Gamma_{\beta}/\{\nu\}}I_{\alpha}[\lambda_{\alpha}(t_{n},\widetilde
{\mathbf{y}}_{n})]_{t_{n},t_{n+1}}){\LARGE
|}\mathcal{F}_{t_{n}}\right) \right\vert ,
\]%
\[
e_{2}=\left\vert E(\mathbf{F}_{\mathbf{p}}(\widetilde{\mathbf{y}}%
_{n+1}-\widetilde{\mathbf{y}}_{n})-\mathbf{F}_{\mathbf{p}}(\widetilde
{\mathbf{z}}_{n+1}-\widetilde{\mathbf{y}}_{n}){\LARGE |}\mathcal{F}_{t_{n}%
})\right\vert
\]
and $\lambda_{\alpha}$ denotes the Ito coefficient function
corresponding to the SDE (\ref{WSDE-LLA-1}). Then, by applying Lemma
5.11.7 in \cite{Kloeden
1995} to the equation (\ref{WSDE-LLA-23}) it is obtained%
\[
\left\vert E\left(  \mathbf{F}_{\mathbf{p}}(\widetilde{\mathbf{z}}%
_{n+1}-\widetilde{\mathbf{y}}_{n})-\mathbf{F}_{\mathbf{p}}(\sum\limits_{\alpha
\in\Gamma_{\beta}/\{\nu\}}I_{\alpha}[\Lambda_{\alpha}(t_{n},\widetilde
{\mathbf{y}}_{n})]_{t_{n},t_{n+1}}){\LARGE
|}\mathcal{F}_{t_{n}}\right) \right\vert \leq K(1+\left\vert
\widetilde{\mathbf{y}}_{n}\right\vert ^{2r})h_{n}^{\beta+1},
\]
which by Lemma \ref{LemmaLL} is equivalent to%
\[
e_{1}\leq K(1+\left\vert \widetilde{\mathbf{y}}_{n}\right\vert ^{2r}%
)h_{n}^{\beta+1},
\]
where $\Lambda_{\alpha}$ denotes the Ito coefficient function
corresponding to the SDE (\ref{WSDE-LLA-23}). From Lemma 10 in
\cite{Carbonell06}, inequalities (\ref{WSDE-LLA-5}) and
(\ref{WSDE-LLA-6}), and the Cauchy-Buniakovski
inequality follows that%
\begin{align*}
e_{2} &  \leq(E\left(  \left\vert
\widetilde{\mathbf{z}}_{n+1}-\widetilde
{\mathbf{y}}_{n+1}\right\vert ^{2}{\LARGE
|}\mathcal{F}_{t_{n}}\right)
)^{1/2}\\
&  \cdot\sum\limits_{j=0}^{l(\mathbf{p})-1}(E\left(  \left\vert
\widetilde
{\mathbf{z}}_{n+1}-\widetilde{\mathbf{y}}_{n}\right\vert ^{4j}{\LARGE |}%
\mathcal{F}_{t_{n}}\right)  )^{1/4}(E\left(  \left\vert \widetilde{\mathbf{y}%
}_{n+1}-\widetilde{\mathbf{y}}_{n}\right\vert ^{l(\mathbf{p})-1-j}%
{\LARGE |}\mathcal{F}_{t_{n}}\right)  )^{1/4}\\
&  \leq K_{6}(1+\left\vert \widetilde{\mathbf{y}}_{n}\right\vert ^{2r_{1}%
})(E\left(  \left\vert \widetilde{\mathbf{z}}_{n+1}-\widetilde{\mathbf{y}%
}_{n+1}\right\vert ^{2}{\LARGE |}\mathcal{F}_{t_{n}}\right) )^{1/2},
\end{align*}
where $K_{6}$ is positive constant and $r_{1}$ a natural number. By
using the expression (\ref{YTruncAux}), inequality
(\ref{WSDE-LLA-21}) and conditions
(\ref{WSDE-LLA-9})-(\ref{WSDE-LLA-10}) follows that%
\begin{align*}
E\left(  \left\vert \widetilde{\mathbf{z}}_{n+1}-\widetilde{\mathbf{y}}%
_{n+1}\right\vert ^{2}{\LARGE |}\mathcal{F}_{t_{n}}\right)   &
\leq2E\left(
\left\vert \mathbf{\phi}_{\beta}(t_{n},\widetilde{\mathbf{y}}_{n}%
;h_{n})-\widetilde{\mathbf{\phi}}_{\beta}(t_{n},\widetilde{\mathbf{y}}%
_{n};h_{n})\right\vert ^{2}{\LARGE |}\mathcal{F}_{t_{n}}\right) \\
& +2E\left(  \left\vert (\mathbf{\Sigma}(t_{n},\widetilde{\mathbf{y}}%
_{n};h_{n})^{1/2}-\widetilde{\mathbf{\Sigma}}(t_{n},\widetilde{\mathbf{y}}%
_{n};h_{n})^{1/2})\xi_{n+1}\right\vert ^{2}{\LARGE |}\mathcal{F}_{t_{n}%
}\right) \\
& \leq2K^{2}(1+C^{2})(1+\left\vert
\widetilde{\mathbf{y}}_{n}\right\vert
^{2})h_{n}^{2\min\{\alpha+1,\gamma+1\}},
\end{align*}
and so%
\[
e_{2}\leq K_{7}(1+\left\vert \widetilde{\mathbf{y}}_{n}\right\vert ^{2r_{2}%
})h_{n}^{\min\{\alpha+1,\gamma+1\}},
\]
where $K_{7}$ is positive constant and $r_{2}$ a natural number.
Hence,
\begin{align*}
& \left\vert E\left(  \mathbf{F}_{\mathbf{p}}(\widetilde{\mathbf{y}}%
_{n+1}-\widetilde{\mathbf{y}}_{n})-\mathbf{F}_{\mathbf{p}}(\sum\limits_{\alpha
\in\Gamma_{\beta}/\{\nu\}}I_{\alpha}[\lambda_{\alpha}(t_{n},\widetilde
{\mathbf{y}}_{n})]_{t_{n},t_{n+1}}){\LARGE
|}\mathcal{F}_{t_{n}}\right)
\right\vert \\
& \leq(K+K_{7})(1+\left\vert \widetilde{\mathbf{y}}_{n}\right\vert ^{2r}%
)h_{n}^{1+\min\{\alpha,\beta,\gamma\}},
\end{align*}
where $r$ is a natural number.

The proof concludes by applying Theorem 14.5.2 in \cite{Kloeden
1995}.
\end{proof}

In order to show the application of previous theorems let us
consider the numerical implementation of the LL discretization
(\ref{Scheme2}) by means of the Pad\'{e} approximation with the
"scaling and squaring" procedure \cite{Moler 2003}.

\begin{theorem}
\label{Theorem Conv LL Pade SDE}Let
\[
\widetilde{\mathbf{B}}(t_{n},\widetilde{\mathbf{y}}_{n};h_{n})=(\mathbf{P}%
_{p,q}(2^{-k_{n}}\mathbf{A}_{\beta}(t_{n},\widetilde{\mathbf{y}}_{n}%
)h_{n}))^{2^{k_{n}}},
\]
where
$\mathbf{P}_{p,q}(2^{-k_{n}}\mathbf{A}_{\beta}(t_{n},\widetilde
{\mathbf{y}}_{n})h_{n})$ denotes the $(p,q)$-Pad\'{e} approximation
of
$e^{2^{-k_{n}}\mathbf{A}_{\beta}(t_{n},\widetilde{\mathbf{y}}_{n})h_{n}}$,
$\mathbf{A}_{\beta}(t_{n},\widetilde{\mathbf{y}}_{n})$ the matrix
defined in (\ref{matrix B WLL scheme}), and $k_{n}$ the smallest
integer number such that
$\left\vert 2^{-k}\mathbf{A}_{\mathbf{\beta}}(t_{n},\widetilde{\mathbf{y}}%
_{n})h_{n}\right\vert \leq\frac{1}{2}$. If the drift and diffusion
coefficients of (\ref{WSDE-LLA-1}) are of class
$\mathcal{C}_{P}^{2(\beta+1)}$ and have uniformly bounded second
derivatives, then the error of the weak LL
scheme%
\begin{equation}
\widetilde{\mathbf{y}}_{n+1}=\widetilde{\mathbf{y}}_{n}+\widetilde{\mathbf{B}%
}_{1,2\beta+2}(t_{n},\widetilde{\mathbf{y}}_{n};h_{n})+\mathbf{(}%
\widetilde{\mathbf{B}}_{1,2\beta}\mathbf{(}t_{n},\widetilde{\mathbf{y}}%
_{n};h_{n}\mathbf{)}\widetilde{\mathbf{B}}_{1,1}^{\intercal}\mathbf{(}%
t_{n},\widetilde{\mathbf{y}}_{n};h_{n}\mathbf{))}^{1/2}\xi_{n+1}%
,\label{WLLS Pade NoAuto}%
\end{equation}
is given by
\[
\left\vert E\left(  g(\mathbf{x}(T))\right)  -E(g(\widetilde{\mathbf{y}%
}_{n_{T}}))\right\vert \leq C_{g}h^{\min\{\beta,p+q\}},
\]
for all
$g\in\mathcal{C}_{P}^{2(\beta+1)}(\mathbb{R}^{d},\mathbb{R})$, where
$\mathbf{x}$ is the solution of
(\ref{WSDE-LLA-1})-(\ref{WSDE-LLA-2}) and $C_{g}$ is a positive
constant.
\end{theorem}

\begin{proof}
Let $\mathbf{L}_{1}$, $\mathbf{R}_{2\beta+2}$, $\mathbf{R}_{2\beta}$
and $\mathbf{R}_{1}$ be matrices such that
\[
\mathbf{B}_{1,2\beta+2}(t_{n},\widetilde{\mathbf{y}}_{n};h_{n})=\mathbf{L}%
_{1}\mathbf{B}(t_{n},\widetilde{\mathbf{y}}_{n};h_{n})\mathbf{R}_{2\beta+2},
\]%
\[
\mathbf{B}_{1,2\beta}(t_{n},\widetilde{\mathbf{y}}_{n};h_{n})=\mathbf{L}%
_{1}\mathbf{B}(t_{n},\widetilde{\mathbf{y}}_{n};h_{n})\mathbf{R}_{2\beta},
\]
and%
\[
\mathbf{B}_{1,1}(t_{n},\widetilde{\mathbf{y}}_{n};h_{n})=\mathbf{L}%
_{1}\mathbf{B}(t_{n},\widetilde{\mathbf{y}}_{n};h_{n})\mathbf{R}_{1},
\]
where the matrix
$\mathbf{B}(t_{n},\widetilde{\mathbf{y}}_{n};h_{n})$ is defined in
(\ref{matrix B WLL scheme}).

Lemma 9 in \cite{Jimenez11} implies that
\begin{align*}
\left\vert
\mathbf{B}(t_{n},\widetilde{\mathbf{y}}_{n};h_{n})-\widetilde
{\mathbf{B}}(t_{n},\widetilde{\mathbf{y}}_{n};h_{n})\right\vert  &
=\left\vert
e^{\mathbf{A}_{\mathbf{\beta}}(t_{n},\widetilde{\mathbf{y}}_{n})h_{n}%
}-(\mathbf{P}_{p,q}(2^{-k_{n}}\mathbf{A}_{\mathbf{\beta}}(t_{n},\widetilde
{\mathbf{y}}_{n})h_{n}))^{2^{k_{n}}}\right\vert \\
& \leq c_{p,q}(k_{n},\left\vert \mathbf{A}_{\mathbf{\beta}}(t_{n}%
,\widetilde{\mathbf{y}}_{n})\right\vert )\left\vert
\mathbf{A}_{\mathbf{\beta
}}(t_{n},\widetilde{\mathbf{y}}_{n})\right\vert
^{p+q+1}h_{n}^{p+q+1},
\end{align*}
where $c_{p,q}(k,\left\vert \mathbf{X}\right\vert )=\alpha2^{-k(p+q)+3}%
e^{(1+\epsilon_{p,q})\left\vert \mathbf{X}\right\vert }$ with
$\alpha
=\frac{p!q!}{(p+q)!(p+q+1)!}$ and $\epsilon_{p,q}=\alpha(\frac{1}{2})^{p+q-3}%
$. Since the drift and diffusion coefficients of (\ref{WSDE-LLA-1})
have uniformly bounded second derivatives, there exists a positive
constant $K_{1}$
such that $\left\vert \mathbf{A}_{\mathbf{\beta}}(t_{n},\widetilde{\mathbf{y}%
}_{n})\right\vert <K_{1}$ and $\left\vert \mathbf{A}_{\mathbf{\beta}}%
(t_{n},\widetilde{\mathbf{y}}_{n})\right\vert \leq
K_{1}(1+\left\vert \widetilde{\mathbf{y}}_{n}\right\vert )$ for all
$n$, which implies that $\left\vert
\mathbf{B}(t_{n},\widetilde{\mathbf{y}}_{n};h_{n})\right\vert
<\infty$ and $\left\vert \widetilde{\mathbf{B}}(t_{n},\widetilde{\mathbf{y}%
}_{n};h_{n})\right\vert <\infty$ for all $n$.

Thus, from these two bounds for $\mathbf{A}_{\mathbf{\beta}}(t_{n}%
,\widetilde{\mathbf{y}}_{n})$ it is obtained that%
\begin{align}
\left\vert \mathbf{\phi_{\beta}}\left(  t_{n},\widetilde{\mathbf{y}}_{n}%
;h_{n}\right)
-\widetilde{\mathbf{\phi}}\mathbf{_{\mathbb{\beta}}}\left(
t_{n},\widetilde{\mathbf{y}}_{n};h_{n}\right)  \right\vert  &
=\left\vert
\mathbf{L}_{1}(\mathbf{B}(t_{n},\widetilde{\mathbf{y}}_{n};h_{n}%
)-\widetilde{\mathbf{B}}(t_{n},\widetilde{\mathbf{y}}_{n};h_{n}))\mathbf{R}%
_{2\beta+2}\right\vert \nonumber\\
& \leq K_{2}(1+\left\vert \widetilde{\mathbf{y}}_{n}\right\vert )h_{n}%
^{p+q+1},\label{WSDE-S1}%
\end{align}
where $K_{2}=c_{p,q}(0,K_{1})\left\vert \mathbf{L}_{1}\right\vert
\left\vert \mathbf{R}_{2\beta+2}\right\vert K_{1}^{p+q+1}$.
Similarly, it is obtained
that%
\begin{align*}
&  \left\vert \mathbf{\Sigma}_{\beta}(t_{n},\widetilde{\mathbf{y}}_{n}%
;h_{n})-\widetilde{\mathbf{\Sigma}}_{\beta}(t_{n},\widetilde{\mathbf{y}}%
_{n};h_{n})\right\vert \\
&  =\left\vert \mathbf{B}_{1,2\beta}(t_{n},\widetilde{\mathbf{y}}_{n}%
;h_{n})\mathbf{B}_{1,1}^{\intercal}(t_{n},\widetilde{\mathbf{y}}_{n}%
;h_{n})-\widetilde{\mathbf{B}}_{1,2\beta}(t_{n},\widetilde{\mathbf{y}}%
_{n};h_{n})\widetilde{\mathbf{B}}_{1,1}^{\intercal}(t_{n},\widetilde
{\mathbf{y}}_{n};h_{n})\right\vert \\
&  \leq\left\vert \mathbf{L}_{1}(\mathbf{B}(t_{n},\widetilde{\mathbf{y}}%
_{n};h_{n})-\widetilde{\mathbf{B}}(t_{n},\widetilde{\mathbf{y}}_{n}%
;h_{n}))\mathbf{R}_{2\beta}\mathbf{R}_{1}^{\intercal}\mathbf{B}^{\intercal
}(t_{n},\widetilde{\mathbf{y}}_{n};h_{n})\mathbf{L}_{1}^{\intercal}\right\vert
\\
&  +\left\vert \mathbf{L}_{1}\widetilde{\mathbf{B}}(t_{n},\widetilde
{\mathbf{y}}_{n};h_{n})\mathbf{R}_{2\beta}\mathbf{R}_{1}^{\intercal
}(\mathbf{B}(t_{n},\widetilde{\mathbf{y}}_{n};h_{n})-\widetilde{\mathbf{B}%
}(t_{n},\widetilde{\mathbf{y}}_{n};h_{n}))^{\intercal}\mathbf{L}%
_{1}^{\intercal}\right\vert \\
&  \leq\left\vert \mathbf{L}_{1}\right\vert ^{2}\left\vert
\mathbf{R}_{2\beta
}\right\vert \left\vert \mathbf{R}_{1}\right\vert \left\vert \mathbf{B}%
(t_{n},\widetilde{\mathbf{y}}_{n};h_{n})-\widetilde{\mathbf{B}}(t_{n}%
,\widetilde{\mathbf{y}}_{n};h_{n})\right\vert \\
&  \cdot(\left\vert \mathbf{B}(t_{n},\widetilde{\mathbf{y}}_{n};h_{n}%
)\right\vert +\left\vert \widetilde{\mathbf{B}}(t_{n},\widetilde{\mathbf{y}%
}_{n};h_{n})\right\vert )\\
&  \leq K_{3}(1+\left\vert \widetilde{\mathbf{y}}_{n}\right\vert
)h_{n}^{p+q+1},
\end{align*}
where $K_{3}$ is a positive constant.

On the other hand, by using the expressions for $\mathbf{\Sigma}$
and $\mathbf{\Sigma}_{\beta}$ and taking into account that the drift
and diffusion coefficients of (\ref{WSDE-LLA-1}) have uniformly
bounded second derivatives
follows that%
\begin{align*}
&  \left\vert \mathbf{\Sigma}(t_{n},\widetilde{\mathbf{y}}_{n};h_{n}%
)-\mathbf{\Sigma}_{\beta}(t_{n},\widetilde{\mathbf{y}}_{n};h_{n})\right\vert
\\
&  \leq Ch_{n}\underset{s\in\lbrack0,h_{n}]}{\sup}\left\vert \mathbf{G}%
(t_{n}+s)\mathbf{G}^{\intercal}(t_{n}+s)-\mathbf{G}_{\beta}(t_{n}%
+s)\mathbf{G}_{\beta}^{\intercal}(t_{n}+s)\right\vert \\
&  \leq Ch_{n}\underset{s\in\lbrack0,h_{n}]}{\sup}(\left\vert \mathbf{G}%
(t_{n}+s)\right\vert +\left\vert
\mathbf{G}_{\beta}(t_{n}+s)\right\vert
)\left\vert \mathbf{G}(t_{n}+s)-\mathbf{G}_{\beta}(t_{n}+s)\right\vert \\
&  \leq C(1+\left\vert \widetilde{\mathbf{y}}_{n}\right\vert
)h_{n}^{\beta+1},
\end{align*}
where $C$ is a positive constant.

From the last two inequalities follows that
\begin{align}
\left\vert \mathbf{\Sigma}(t_{n},\widetilde{\mathbf{y}}_{n};h_{n}%
)-\widetilde{\mathbf{\Sigma}}(t_{n},\widetilde{\mathbf{y}}_{n};h_{n}%
)\right\vert  &  \leq\left\vert \mathbf{\Sigma}(t_{n},\widetilde{\mathbf{y}%
}_{n};h_{n})-\mathbf{\Sigma}_{\beta}(t_{n},\widetilde{\mathbf{y}}_{n}%
;h_{n})\right\vert \nonumber\\
&  +\left\vert \mathbf{\Sigma}_{\beta}(t_{n},\widetilde{\mathbf{y}}_{n}%
;h_{n})-\widetilde{\mathbf{\Sigma}}(t_{n},\widetilde{\mathbf{y}}_{n}%
;h_{n})\right\vert \nonumber\\
&  \leq(C+K_{3})(1+\left\vert \widetilde{\mathbf{y}}_{n}\right\vert
)h_{n}^{\min\{\beta+1,p+q+1\}},\label{WSDE-S2}%
\end{align}

The proof concludes by using Theorem \ref{WSCT} with inequalities
(\ref{WSDE-S1})-(\ref{WSDE-S2}).
\end{proof}

Similarly for SDEs with constant diffusion coefficients, it can be
proved that
the error of the weak LL scheme%
\begin{equation}
\widetilde{\mathbf{y}}_{n+1}=\widetilde{\mathbf{y}}_{n}+\widetilde{\mathbf{D}%
}_{14}(t_{n},\widetilde{\mathbf{y}}_{n};h_{n})+\mathbf{(}\widetilde
{\mathbf{D}}_{12}\mathbf{(}t_{n},\widetilde{\mathbf{y}}_{n};h_{n}%
\mathbf{)}\widetilde{\mathbf{D}}_{11}^{\intercal}\mathbf{(}t_{n}%
,\widetilde{\mathbf{y}}_{n};h_{n}\mathbf{))}^{1/2}\xi_{n+1},\label{WLLS Pade}%
\end{equation}
obtained from (\ref{Scheme1}) with $\widetilde{\mathbf{D}}(t_{n}%
,\widetilde{\mathbf{y}}_{n};h_{n})=(\mathbf{P}_{p,q}(2^{-k_{n}}\mathbf{C}%
_{\beta}(t_{n},\widetilde{\mathbf{y}}_{n})h_{n}))^{2^{k_{n}}}$, is
given by
\[
\left\vert E\left(  g(\mathbf{x}(T))\right)  -E\left(  g(\widetilde
{\mathbf{y}}_{n_{T}})\right)  \right\vert \leq C_{g}h^{\min\{\beta,p+q\}}%
\]
for all
$g\in\mathcal{C}_{P}^{2(\beta+1)}(\mathbb{R}^{d},\mathbb{R})$, where
$C_{g}$ is a positive constant.

We recall from \cite{de la Cruz HOLL} that LL schemes (\ref{WLLS
Pade NoAuto}) and (\ref{WLLS Pade}) are $A$-stable, therefore they
preserve the ergodicity of the linear SDEs. They also are
geometrically ergodic for some class of nonlinear SDEs
\cite{Hansen03}. However, due to the use of Pad\'{e} approximations
these schemes are not appropriate for large dimensional systems of
SDEs. For that type of equations, LL schemes based on Krylov method
for matrix exponential \cite{Moler 2003} are recommended.

In such a case, the LL discretization (\ref{Scheme1}) can be
rewritten as
\[
\mathbf{y}_{n+1}=\mathbf{y}_{n}+\mathbf{P}^{\intercal}(t_{n},\mathbf{y}%
_{n};h_{n})\mathbf{R}_{4}+\mathbf{(P^{\intercal}}(t_{n},\mathbf{y}_{n}%
;h_{n})\mathbf{R}_{2}\mathbf{R}_{1}^{\intercal}\mathbf{P(}t_{n},\mathbf{y}%
_{n};h_{n}\mathbf{))}^{1/2}\xi_{n+1}%
\]
where
$\mathbf{P(}t_{n},\mathbf{y}_{n};h_{n}\mathbf{)}=e^{\mathbf{C}_{\beta
}^{\intercal}(t_{n},\mathbf{y}_{n})h_{n}}\mathbf{L}_{1}^{\intercal}$
and $\mathbf{L}_{1},\mathbf{R}_{1},\mathbf{R}_{2},\mathbf{R}_{4}$
are matrices such that
\[
\mathbf{D}_{1,4}(t_{n},\mathbf{y}_{n};h_{n})=\mathbf{L}_{1}e^{\mathbf{C}%
_{\beta}(t_{n},\mathbf{y}_{n})h_{n}}\mathbf{R}_{4},
\]%
\[
\mathbf{D}_{1,2}(t_{n},\mathbf{y}_{n};h_{n})=\mathbf{L}_{1}e^{\mathbf{C}%
_{\beta}(t_{n},\mathbf{y}_{n})h_{n}}\mathbf{R}_{2},
\]
and%
\[
\mathbf{D}_{1,1}(t_{n},\mathbf{y}_{n};h_{n})=\mathbf{L}_{1}e^{\mathbf{C}%
_{\beta}(t_{n},\mathbf{y}_{n})h_{n}}\mathbf{R}_{1}.
\]
If the Krylov-Pad\'{e} method is used to compute
\[
vec(\mathbf{P(}t_{n},\widetilde{\mathbf{y}}_{n};h_{n}))=e^{\mathbf{I}%
\otimes\mathbf{C}_{\beta}^{\intercal}(t_{n},\widetilde{\mathbf{y}}_{n})h_{n}%
}vec(\mathbf{L}_{1}^{\intercal}),
\]
we have the following LL scheme%
\begin{equation}
\widetilde{\mathbf{y}}_{n+1}=\widetilde{\mathbf{y}}_{n}+\widetilde{\mathbf{P}%
}^{\intercal}(t_{n},\widetilde{\mathbf{y}}_{n};h_{n})\mathbf{R}_{4}%
+\mathbf{(\widetilde{\mathbf{P}}^{\intercal}}(t_{n},\widetilde{\mathbf{y}}%
_{n};h_{n})\mathbf{R}_{2}\mathbf{R}_{1}^{\intercal}\widetilde{\mathbf{P}%
}\mathbf{(}t_{n},\widetilde{\mathbf{y}}_{n};h_{n}\mathbf{))}^{1/2}\xi
_{n+1},\label{WLLS Krylov}%
\end{equation}
where
\[
vec(\widetilde{\mathbf{P}}\mathbf{(}t_{n},\widetilde{\mathbf{y}}_{n}%
;h_{n}\mathbf{)})=\mathbf{k}_{m_{n},k_{n}}^{p,q}(h_{n},\mathbf{I}%
\otimes\mathbf{C}_{\beta}^{\intercal}(t_{n},\widetilde{\mathbf{y}}%
_{n}),vec(\mathbf{L}_{1}^{\intercal})),
\]
$\mathbf{k}_{m_{n},k_{n}}^{p,q}$ denotes the $(m_{n},p,q,k_{n})-$%
Krylov-Pad\'{e} approximation defined as in \cite{Jimenez11}, and
$\mathbf{I}$ is the identity matrix of dimension $2d+2$.

At glance, this numerical scheme seems to be computationally
inefficient since it involves the computation of large matrix
exponentials. Indeed,
$\mathbf{I}\otimes\mathbf{C}_{\beta}^{\intercal}(t_{n},\widetilde{\mathbf{y}%
}_{n})$ is a $(2d+2)^{2}\times(2d+2)^{2}$ matrix. However, this
matrix is
block diagonal with diagonal entries $\mathbf{C}_{\beta}^{\intercal}%
(t_{n},\widetilde{\mathbf{y}}_{n})$. This block structure allows us
to save computer storage capacity with an adequate algorithmic
implementation. In addition, it implies that the number $m_{n}$ of
Krylov subspaces necessary to compute
$e^{(\frac{1}{h_{n}}\mathbf{I}\otimes\mathbf{C}_{\beta}^{\intercal
}(t_{n},\widetilde{\mathbf{y}}_{n}))h_{n}}vec(\mathbf{L}_{1}^{\intercal})$
has
the same order of magnitude than that needed for computing $e^{\mathbf{C}%
_{\beta}^{\intercal}(t_{n},\widetilde{\mathbf{y}}_{n})h_{n}}\mathbf{1}$.
Typically, $m_{n}<<d$ in practical situations. This makes the LL
scheme (\ref{WLLS Krylov}) feasible and computationally efficient.

\begin{theorem}
Let $\mathbf{x}$ be the solution of a SDE with constant diffusion
coefficients and drift coefficient of class
$\mathcal{C}_{P}^{2(\beta+1)}$ with uniformly bounded second
derivatives. If $m_{n}\geq2h_{n}\left\vert \mathbf{C}_{\beta
}(t_{n},\widetilde{\mathbf{y}}_{n})\right\vert _{2}$ for all $n$,
then the error of the weak LL scheme (\ref{WLLS Krylov}) is given by
\[
\left\vert E\left(  g(\mathbf{x}(T))\right)  -E(g(\widetilde{\mathbf{y}%
}_{n_{T}}))\right\vert _{2}\leq C_{g}h^{\min\{\beta,m-1,p+q\}},
\]
for all
$g\in\mathcal{C}_{P}^{2(\beta+1)}(\mathbb{R}^{d},\mathbb{R})$, where
$m=\min\{m_{n}\}$, $C_{g}$ is a positive constant, and $\left\vert
\cdot\right\vert _{2}$ denotes the Euclidean norm.
\end{theorem}

\begin{proof}
Taking into account that $\left\vert
\mathbf{I}\otimes\mathbf{C}_{\beta
}^{\intercal}(t_{n},\widetilde{\mathbf{y}}_{n})\right\vert
_{2}=\left\vert
\mathbf{C}_{\beta}^{\intercal}(t_{n},\widetilde{\mathbf{y}}_{n})\right\vert
_{2}$, Lemma 11 in \cite{Jimenez11} implies that
\begin{align*}
&  \left\vert
\mathbf{P}(t_{n},\widetilde{\mathbf{y}}_{n};h_{n})-\widetilde
{\mathbf{P}}(t_{n},\widetilde{\mathbf{y}}_{n};h_{n})\right\vert _{2}\\
&  =\left\vert e^{\mathbf{C}_{\beta}^{\intercal}(t_{n},\widetilde{\mathbf{y}%
}_{n})h_{n}}\mathbf{L}_{1}^{\intercal}-\mathbf{k}_{m_{n},k_{n}}^{p,q}%
(h_{n},\mathbf{I}\otimes\mathbf{C}_{\beta}^{\intercal}(t_{n},\widetilde
{\mathbf{y}}_{n}),vec(\mathbf{L}_{1}^{\intercal}))\right\vert _{2}\\
&  \leq C_{m_{n},k_{n}}^{p,q}(1,\left\vert
\mathbf{C}_{\beta}^{\intercal
}(t_{n},\widetilde{\mathbf{y}}_{n})\right\vert _{2})\left\vert h_{n}%
\mathbf{C}_{\beta}^{\intercal}(t_{n},\widetilde{\mathbf{y}}_{n})\right\vert
_{2}^{\min\{m_{n},p+q+1\}},
\end{align*}
where $C_{m,\kappa}^{p,q}(\beta,\rho)=\beta$ $c_{p,q}(\kappa,\rho)$
$+12\beta e^{m-\rho}(\frac{1}{m})^{m}$ with $c_{p,q}(k,\left\vert
\mathbf{X}\right\vert
)=\alpha2^{-k(p+q)+3}e^{(1+\epsilon_{p,q})\left\vert
\mathbf{X}\right\vert }$ and $\alpha=\frac{p!q!}{(p+q)!(p+q+1)!}$.
Since the diffusion coefficients are constants and the drift
coefficient has uniformly bounded second derivatives, there exists a
positive constant $M$ such that $\left\vert \mathbf{C}_{\beta
}^{\intercal}(t_{n},\widetilde{\mathbf{y}}_{n})\right\vert _{2}<M$
and
$\left\vert \mathbf{C}_{\beta}^{\intercal}(t_{n},\widetilde{\mathbf{y}}%
_{n})\right\vert _{2}\leq M(1+\left\vert
\widetilde{\mathbf{y}}_{n}\right\vert )$ for all $n$, which implies
that $\left\vert \mathbf{P}(t_{n},\widetilde
{\mathbf{y}}_{n};h_{n})\right\vert _{2}<\infty$ and $\left\vert
\widetilde
{\mathbf{P}}(t_{n},\widetilde{\mathbf{y}}_{n};h_{n})\right\vert
_{2}<\infty$ for all $n$.

Thus, from these two bounds for $\mathbf{C}_{\beta}^{\intercal}(t_{n}%
,\widetilde{\mathbf{y}}_{n})$ it is obtained that%
\begin{align*}
\left\vert \mathbf{\phi_{\beta}}\left(  t_{n},\widetilde{\mathbf{y}}_{n}%
;h_{n}\right)
-\widetilde{\mathbf{\phi}}\mathbf{_{\mathbb{\beta}}}\left(
t_{n},\widetilde{\mathbf{y}}_{n};h_{n}\right)  \right\vert _{2}  &
=\left\vert
\mathbf{P}^{\intercal}(t_{n},\widetilde{\mathbf{y}}_{n};h_{n})\mathbf{R}%
_{4}-\widetilde{\mathbf{P}}^{\intercal}(t_{n},\widetilde{\mathbf{y}}_{n}%
;h_{n})\mathbf{R}_{4}\right\vert _{2}\\
& \leq M_{1}(1+\left\vert \widetilde{\mathbf{y}}_{n}\right\vert _{2}%
^{2})^{1/2}h_{n}^{\min\{m,p+q+1\}},
\end{align*}
where $m=\min\{m_{n}\}$ and $M_{1}=C_{m,0}^{p,q}(1,M)\left\vert \mathbf{R}%
_{4}\right\vert M^{\min\{m-1,p+q\}}$. Similarly, it is obtained that%
\begin{align*}
&  \left\vert \mathbf{\Sigma}_{\beta}(t_{n},\widetilde{\mathbf{y}}_{n}%
;h_{n})-\widetilde{\mathbf{\Sigma}}(t_{n},\widetilde{\mathbf{y}}_{n}%
;h_{n})\right\vert _{2}\\
&  =\left\vert \mathbf{\mathbf{P}^{\intercal}}(t_{n},\widetilde{\mathbf{y}%
}_{n};h_{n})\mathbf{R}_{2}\mathbf{R}_{1}^{\intercal}\mathbf{P(}t_{n}%
,\widetilde{\mathbf{y}}_{n};h_{n}\mathbf{)}-\mathbf{\widetilde{\mathbf{P}%
}^{\intercal}}(t_{n},\widetilde{\mathbf{y}}_{n};h_{n})\mathbf{R}_{2}%
\mathbf{R}_{1}^{\intercal}\widetilde{\mathbf{P}}\mathbf{(}t_{n},\widetilde
{\mathbf{y}}_{n};h_{n}\mathbf{)}\right\vert _{2}\\
&  \leq\left\vert (\mathbf{\mathbf{P}^{\intercal}}(t_{n},\widetilde
{\mathbf{y}}_{n};h_{n})-\widetilde{\mathbf{\mathbf{P}}}\mathbf{^{\intercal}%
}(t_{n},\widetilde{\mathbf{y}}_{n};h_{n}))\mathbf{R}_{2}\mathbf{R}%
_{1}^{\intercal}\mathbf{P(}t_{n},\widetilde{\mathbf{y}}_{n};h_{n}%
\mathbf{)}\right\vert _{2}\\
&  +\left\vert \widetilde{\mathbf{\mathbf{P}}}\mathbf{^{\intercal}}%
(t_{n},\widetilde{\mathbf{y}}_{n};h_{n}))\mathbf{R}_{2}\mathbf{R}%
_{1}^{\intercal}(\mathbf{\mathbf{P}}(t_{n},\widetilde{\mathbf{y}}_{n}%
;h_{n})-\widetilde{\mathbf{\mathbf{P}}}(t_{n},\widetilde{\mathbf{y}}_{n}%
;h_{n}))\right\vert _{2}\\
&  \leq\left\vert \mathbf{R}_{2}\right\vert _{2}\left\vert \mathbf{R}%
_{1}^{\intercal}\right\vert _{2}\left\vert \mathbf{\mathbf{P}}(t_{n}%
,\widetilde{\mathbf{y}}_{n};h_{n})-\widetilde{\mathbf{\mathbf{P}}}%
(t_{n},\widetilde{\mathbf{y}}_{n};h_{n})\right\vert _{2}\\
&  \cdot(\left\vert \mathbf{P}(t_{n},\widetilde{\mathbf{y}}_{n};h_{n}%
)\right\vert _{2}+\left\vert \widetilde{\mathbf{P}}(t_{n},\widetilde
{\mathbf{y}}_{n};h_{n})\right\vert _{2})\\
&  \leq M_{2}(1+\left\vert \widetilde{\mathbf{y}}_{n}\right\vert _{2}%
^{2})^{1/2}h_{n}^{\min\{m,p+q+1\}},
\end{align*}
where $M_{2}$ is a positive constant.

The proof concludes by using Theorem \ref{WSCT} with last two
inequalities.
\end{proof}

In the theorem above the restriction to the norm $\left\vert
\cdot\right\vert _{2}$ results from the condition
$m_{n}\geq2h_{n}\left\vert \mathbf{C}_{\beta
}(t_{n},\widetilde{\mathbf{y}}_{n})\right\vert _{2}$, which is
required to establish the convergence of the Krylov-Pad\'{e}
approximation to the exponential matrices. Nevertheless, depending
on the class of the matrix
$\mathbf{C}_{\beta}(t_{n},\widetilde{\mathbf{y}}_{n})$ and/or the
location and shape of its spectrum (see, e.g., \cite{Sidje 1998} and
references therein), such restriction might be discarded.

Analogously, for nonautonomous SDEs, the discretization
(\ref{Scheme2}) can be
rewritten in terms of $\mathbf{P(}t_{n},\mathbf{y}_{n};h_{n}\mathbf{)}%
=e^{\mathbf{A}_{\beta}^{\intercal}(t_{n},\mathbf{y}_{n})h_{n}}\mathbf{L}%
_{1}^{\intercal}$ to obtain a LL scheme similar to (\ref{WLLS
Krylov}) in
terms of $vec(\widetilde{\mathbf{P}}\mathbf{(}t_{n},\widetilde{\mathbf{y}}%
_{n};h_{n}\mathbf{)})=\mathbf{k}_{m_{n},k_{n}}^{p,q}(h_{n},\mathbf{I}%
\otimes\mathbf{A}_{\beta}^{\intercal}(t_{n},\widetilde{\mathbf{y}}%
_{n}),vec(\mathbf{L}_{1}^{\intercal}))$. The convergence of such a
scheme can be then proved as in the previous theorem.

In addition, it is worth noting that, since the LL discretization
(\ref{LLDiscretization}) provides weak solutions for autonomous
linear SDEs with additive noise at all $t_{n}\in(t)_{h}$, the LL
schemes (\ref{WLLS Pade}) and (\ref{WLLS Krylov}) converge to weak
solutions of these equations with order $p+q$ and $\min\{m,p+q\}$,
respectively.

\section{Extension to SDEs with jumps}

Consider a $d$-dimensional jump diffusion process $\mathbf{z}$ defined by the SDE%

\begin{align}
d\mathbf{z}(t)  & =\mathbf{f}(t,\mathbf{z}(t))dt+\sum\limits_{i=1}%
^{m}\mathbf{g}_{i}(t)d\mathbf{w}^{i}(t)+\sum\limits_{i=1}^{p}\mathbf{h}%
_{i}\mathbf{(}t,\mathbf{z}(t))d\mathbf{q}^{i}(t)\label{W Jump SDE1}\\
\mathbf{z}(t_{0})  & =\mathbf{x}_{0},\label{W Jump SDE1b}%
\end{align}
where $\mathbf{f}$, $\mathbf{g}_{i},\mathbf{w}$ are defined as in
(\ref{WSDE-LLA-1}),
$\mathbf{h}_{i}:[t_{0},T]\times\mathbb{R}^{d}\rightarrow$
$\mathbb{R}^{d}$ is a function, and $\mathbf{q}^{i}(t)$ is a $\mathcal{F}_{t}%
$-adapted Poisson counting process $\mathbf{n}^{i}(t)$ with
intensity
$\mathbf{\mu}^{i}$. It is assumed that $\mathbf{w}^{i}(t)$ and $\mathbf{q}%
^{j}(t)$ are all independent with zero probability of simultaneous
jumps for all $t$.

Further, let us consider the sequence of jump times
$\{\mathbb{\sigma}\}_{\mathbb{\mu}^{i}
}=\{\mathbb{\sigma}_{i,n}:n=0,1,2,\ldots\}$ associated to
$\mathbf{q}^{i}(t)$, which is defined as an increasing sequence of
random variables such that
$\mathbb{\sigma}_{i,n+1}-\mathbb{\sigma}_{i,n}$ is exponentially
distributed with parameter $\mu^{i}$, for all $n$ and $i$. It is
assumed that $\{\mathbb{\sigma\}}_{\mathbb{\mu}^{i}}\subset(t)_{h}$
for all $i=1,\ldots,p$, where $(t)_{h}$ is a time discretization
defined as before.
\newline

\begin{definition}
(\cite{Carbonell08}) For a given time discretization $(t)_{h}$, the
order-$\mathbb{\beta}$ weak Local Linear discretization of the jump
diffusion process $\mathbf{z}$ is defined by the recursive relation
\begin{equation}
\mathbf{z}_{n}=\mathbf{z}_{n-}+\sum\limits_{i=1}^{p}\mathbf{h}_{i}%
(t_{n},\mathbf{z}_{n-})\Delta\mathbf{n}_{n}^{i},\label{IL.C.15}%
\end{equation}
where $\mathbf{z}_{n-}$ denotes the value $\mathbf{y}_{n}$ of an
order-$\mathbb{\beta}$ weak Local Linear discretization of diffusion
process defined by (\ref{WSDE-LLA-1}) on $[t_{n-1},t_{n}]$ with
initial condition $\mathbf{x}(t_{n-1})=\mathbf{z}_{n-1}$, and
$\Delta\mathbf{n}_{n}^{i}$ is the increment of the process
$\mathbf{n}_{n}^{i}$ at the time instant $t_{n}$.
\end{definition}

According to the results of the subsection above, it is easy to
realize that numerical implementations of weak Local Linear
discretization for SDE with jumps involve the use of weak LL schemes
for equations with no jumps. Indeed, an weak Local Linearization
scheme $\widetilde{\mathbf{z}}_{n}$ for the integration of the SDE
with jumps (\ref{W Jump SDE1})-(\ref{W Jump SDE1b}) can be defined
as
\begin{equation}
\widetilde{\mathbf{z}}_{n}=\widetilde{\mathbf{z}}_{n-}+\sum\limits_{i=1}%
^{p}\mathbf{h}_{i}(t_{n},\widetilde{\mathbf{z}}_{n-})\Delta\mathbf{n}_{n}%
^{i},\label{WLL scheme for SDEs with jumps}%
\end{equation}
where $\widetilde{\mathbf{z}}_{n-}$ denotes the value $\widetilde{\mathbf{y}%
}_{n}$ an order-$\mathbb{\beta}$ \ weak LL scheme for the SDE
(\ref{WSDE-LLA-1}) on $[t_{n-1},t_{n}]$ with initial condition $\mathbf{z}%
(t_{n-1})=\widetilde{\mathbf{y}}_{n-1}$, and
$\Delta\mathbf{n}_{n}^{i}$ is the increment of the process
$\mathbf{n}_{n}^{i}$ at the time instant $t_{n}$.

In order to study the convergence of the LL scheme (\ref{WLL scheme
for SDEs with jumps}) the following result is useful, which is a
straightforward consequence of Theorem \ref{WSCT}.

\begin{corollary}
\label{Corollary Conv WLL jump}Under conditions of Theorem
\ref{WSCT}, all LL scheme for the SDE
(\ref{WSDE-LLA-1})-(\ref{WSDE-LLA-2}) defined as in (\ref{GWLLS})
satisfies the hypothesis of Theorem 14.5.2 in \cite{Kloeden 1995}.
\end{corollary}

The main convergence result is the following.

\begin{theorem}
Let $\widetilde{\mathbf{z}}_{n}$ be the order-$\mathbb{\beta}$ LL
scheme defined in (\ref{WLL scheme for SDEs with jumps}) for the SDE
with jumps (\ref{W Jump SDE1})-(\ref{W Jump SDE1b}). Suppose that
the functions $\mathbf{h}_{i}$ defined in (\ref{W Jump SDE1})
satisfy the linear growth bound
\[
\left\vert \mathbf{h}_{i}(t,\mathbf{u})\right\vert \leq
K(1+\left\vert \mathbf{u}\right\vert )
\]
for $t\in\lbrack t_{0},T]$ and $\mathbf{u}\in\mathbb{R}^{d}$. Then,
under conditions of Theorem \ref{WSCT}, there exits a positive
constant $C_{g}$ such
that%
\[
\left\vert E\left(  g(\mathbf{z}(T))\right)  -E\left(  g(\widetilde
{\mathbf{z}}_{n_{T}})\right)  \right\vert \leq C_{g}(T-t_{0})h^{\beta}%
\]
for all
$g\in\mathcal{C}_{P}^{2(\beta+1)}(\mathbb{R}^{d},\mathbb{R})$.
\end{theorem}

\begin{proof}
It directly follows from the definition (\ref{WLL scheme for SDEs with jumps}%
), Corollary \ref{Corollary Conv WLL jump}, and Theorem 13.6.1 in
\cite{Platen 2010}.
\end{proof}

Here it is worth to mention that the above definitions and results
can be easily adapted for SDEs driven for nonhomogeneous and/or
compensated Poisson processes as well.

\section{Discussion}

Results of this work provide order conditions for the approximations
to $\mathbf{\phi}_{\beta}$ and $\mathbf{\eta}$ in such a way that
the resulting weak LL scheme preserves the convergence order of the
underlaying LL discretization. This gives, for first time, a clear
guideline for the numerical implementation of computationally
efficient weak LL schemes for SDEs. According to this, it is not
difficult to realize that the available LL schemes are not so
efficient as they could be. Usually, they are obtained from
approximations to $\mathbf{\phi}_{\beta}$ and $\mathbf{\eta}$ with
higher order of convergence than that required, which involve an
unnecessary extra computational cost. Notice that, many of the
available numerical computing environments (as Matlab, Octave, etc)
provide subroutines for the computation of complex matrix operations
up to the precision of the floating-point arithmetic. This includes
the computation of inverse matrix, exponential matrix, and Schur
decomposition required by various LL schemes. Therefore, most of the
weak LL schemes use an "exact " (up to the precision of the
floating-point arithmetic) algorithm to compute
$\mathbf{\phi}_{\beta}$ and $\mathbf{\eta}$. For example, the LL
schemes (\ref{WLLS Pade NoAuto}) currently implemented in Matlab use
the building-in subroutine\ "expm", which implements a high order
Pad\'{e} approximation to exponential matrix. However, according to
the Theorem \ref{Theorem Conv LL Pade SDE}, the lower order
$(1,1)$-Pad\'{e} approximation is sufficient for preserving the
order of convergence of these schemes. In this way, various matrix
multiplications could be saved at each integration step, which
implies a sensible reduction of the overall computational cost. This
points out the important practical value of the results obtained in
this work.

\end{document}